\documentclass[10pt,reqno]{amsart}

\usepackage{amsfonts,amssymb,amsmath,mathrsfs,graphicx}
\usepackage{float}
\usepackage{epsfig}
\usepackage{graphicx}
\usepackage{caption}
\usepackage{subcaption}
\usepackage{color}
\usepackage{kotex}
\usepackage{comment}
\usepackage[hidelinks]{hyperref}

\title[ ]{Uniform position alignment estimate of spherical flocking model with inter-particle bonding forces}

\author[Sun-Ho Choi]{Sun-Ho Choi}
\address[Sun-Ho Choi]{Department of Applied Mathematics and the Institute of Natural Sciences, Kyung Hee University, Yongin, 17104,  Republic of Korea}
\email{sunhochoi@khu.ac.kr}

\author[Dohyun Kwon]{Dohyun Kwon}
\address[Dohyun Kwon]{Department of Mathematics, University of Wisconsin-Madison, 480 Lincoln Dr., Madison, WI 53706, USA}
\email{dkwon7@wisc.edu}

\author[Hyowon Seo]{Hyowon Seo}
\address[Hyowon Seo]{Department of Applied Mathematics and the Institute of Natural Sciences, Kyung Hee University, Yongin, 17104,  Republic of Korea}
\email{hyowseo@gmail.com}

\begin{document}

\newtheorem{theorem}{Theorem} [section]
\newtheorem{maintheorem}{Theorem}
\newtheorem{lemma}[theorem]{Lemma}
\newtheorem{assumption}[theorem]{Assumption}
\newtheorem{proposition}[theorem]{Proposition}
\newtheorem{remark}[theorem]{Remark}
\newtheorem{example}{Example}
\newtheorem{exercise}{Exercise}
\newtheorem{question}{Question}
\newtheorem{definition}{Definition}[section]
\newtheorem{corollary}[theorem]{Corollary}


\def\A{{\mathcal A}}
\def\B{{\mathcal B}}
\def\C{{\mathcal C}}
\def\D{{\mathcal D}}
\def\E{{\mathcal E}}
\def\F{{\mathcal F}}
\def\G{{\mathcal G}}
\def\H{{\mathcal H}}
\def\I{{\mathcal I}}
\def\J{{\mathcal J}}
\def\K{{\mathcal K}}
\def\L{{\mathcal L}}
\def\M{{\mathcal M}}
\def\N{{\mathbb N}}
\def\O{{\mathcal O}}
\def\P{{\mathbb P}}
\def\Q{{\mathbb Q}}
\def\R{{\mathbb R}}
\def\S{{\mathbb S}}
\def\T{{\mathcal T}}
\def\W{{\mathcal W}}
\def\X{{\vec{X}}}
\def\V{{\mathcal V}}
\def\Y{{\mathcal Y}}
\def\Z{{\mathbb Z}}

\newcommand{\noi}{\noindent}
\newcommand{\Rd}{{\R^d}} %
\newcommand{\Rn}{{\R^n}} %
\newcommand{\bul}{\bullet}

\newcommand{\RR}{\mathcal{R}}
\newcommand{\HH}{\mathcal{H}}

\newcommand{\al}{\alpha}
\newcommand{\dl}{\delta}
\newcommand{\Dl}{\Delta}
\newcommand{\eps}{\varepsilon}
\newcommand{\e}{\varepsilon} %
\newcommand{\kk}{\kappa}
\newcommand{\gam}{\gamma}
\newcommand{\Gam}{\Gamma} %
\newcommand{\lam}{\lambda}
\newcommand{\ld}{\lambda}
\newcommand{\Lam}{\Lambda}
\newcommand{\Ld}{\Lambda}
\newcommand{\s}{N\sigma}
\newcommand{\ft}{\widehat}
\newcommand{\wt}{\widetilde}
\newcommand{\cj}{\overline}
\newcommand{\dx}{\partial_x}
\newcommand{\dt}{\partial_t}
\newcommand{\dd}{\partial}
\newcommand{\invft}[1]{\overset{\vee}{#1}}
\newcommand{\lrarrow}{\leftrightarrow}
\newcommand{\embeds}{\hookrightarrow}
\newcommand{\LRA}{\Longrightarrow}
\newcommand{\LLA}{\Longleftarrow}

\newcommand{\om}{\omega}
\newcommand{\Om}{\Omega}

\newcommand{\wto}{\rightharpoonup}

\newcommand{\jb}[1]
{\langle #1 \rangle}

\newcommand{\kwon}[1]{{\color{blue} #1  }}


\renewcommand{\theequation}{\thesection.\arabic{equation}}
\renewcommand{\thetheorem}{\thesection.\arabic{theorem}}
\renewcommand{\thelemma}{\thesection.\arabic{lemma}}
\newcommand{\bbr}{\mathbb R}
\newcommand{\bbz}{\mathbb Z}
\newcommand{\bbn}{\mathbb N}
\newcommand{\bbs}{\mathbb S}
\newcommand{\bbp}{\mathbb P}
\newcommand{\ddiv}{\textrm{div}}
\newcommand{\bn}{\bf n}
\newcommand{\rr}[1]{\rho_{{#1}}}
\newcommand{\thh}{\theta}
\def\charf {\mbox{{\text 1}\kern-.24em {\text l}}}
\renewcommand{\arraystretch}{1.5}

\newcommand{\Rp}{{\mathbb R^+}}
\newcommand{\Rpz}{{[0, \infty)}}
\newcommand{\Rm}{{\mathbb R^+}}
\newcommand{\Rmz}{{\mathbb R^-_0}}
\newcommand{\Rinf}{{\R \cup \{+\infty\}}}
\newcommand{\Rinfm}{{\R \cup \{-\infty\}}}
\newcommand{\Rinfpm}{{\R \cup \{\pm \infty\}}}
\newcommand{\Rpt}{{[0,T]}}


\newcommand{\ro}{{R_{\cdot \rightarrow \cdot}}}
\newcommand{\rji}{{R_{x_j \rightarrow x_i}}}
\newcommand{\rki}{{R_{x_k \rightarrow x_i}}}
\newcommand{\rij}{{R_{x_i \rightarrow x_j}}}
\newcommand{\rjih}{{R_{x_j \rightarrow \frac{x_i+x_j}{|x_i+x_i|}}}}
\newcommand{\rijh}{{R_{x_i \rightarrow \frac{x_i+x_j}{|x_j+x_i|}}}}

\newcommand{\hR}{{P}}
\newcommand{\hro}{{\hR_{\cdot \rightarrow \cdot}}}
\newcommand{\hrji}{{\hR_{x_j \rightarrow x_i}}}
\newcommand{\hrij}{{\hR_{x_i \rightarrow x_j}}}
\newcommand{\hrki}{{\hR_{x_k \rightarrow x_i}}}

\newcommand{\rotji}{{R_{z_j \rightarrow z_i}}}
\newcommand{\rot}{{R_{z_1 \rightarrow z_2}}}
\newcommand{\rott}{{R_{z_1 \rightarrow z_2}^T}}
\newcommand{\roti}{{R_{z_1 \rightarrow z_2}^{-1}}}
\newcommand{\rotr}{{R_{z_2 \rightarrow z_1}}}

\newcommand{\hrot}{{\hR_{z_1 \rightarrow z_2}}}
\newcommand{\hrott}{{\hR_{z_1 \rightarrow z_2}^T}}
\newcommand{\hroti}{{P_{z_1 \rightarrow z_2}^{-1}}}
\newcommand{\hrotr}{{P_{z_2 \rightarrow z_1}}}

\newcommand{\xc}{{x_1\times x_2}}
\newcommand{\xd}{{x_1 \cdot x_2}}
\newcommand{\xcji}{{x_j\times x_i}}

\newcommand{\zc}{{z_1\times z_2}}
\newcommand{\zd}{{z_1 \cdot z_2}}
\newcommand{\zcji}{{z_j\times z_i}}


\newcommand{\sumi}{{\sum_{i=1}^{N}}}
\newcommand{\sumj}{{\sum_{j=1}^{N}}}
\newcommand{\sumij}{{\sum_{i,j=1}^{N}}}
\newcommand{\avei}{{\frac{1}{N}\sum_{i=1}^{N}}}
\newcommand{\xvi}{{\{(x_i, v_i )\}_{\idn}}}
\newcommand{\xii}{{\{x_i\}_{\idn}}}
\newcommand{\vii}{{\{v_i\}_{\idn}}}
\newcommand{\idn}{{1 \leq i \leq N}}
\newcommand{\ijdn}{{1 \leq i,j \leq N}}

\newcommand{\vsup}{{\|v\|_{2,\infty}}}
\newcommand{\vtt}{{\|v\|_{2,2}}}
\newcommand{\vttt}{{\|v\|_{2,2}^2}}
\newcommand{\vssavg}{{\frac{1}{N}\|v\|_{2,2}^2}}

\newcommand{\dox}{{\dot{x}}}
\newcommand{\dov}{{\dot{v}}}
\newcommand{\dow}{{\dot{w}}}

\thanks{
}
\begin{abstract}
We present a sufficient condition of the complete position flocking theorem for the Cucker-Smale type model on the unit sphere with an inter-particle bonding force. For this second order dynamical system derived in [Choi, S.-H.,  Kwon, D. and Seo, H.:  Cucker-Smale type flocking models on a sphere. arXiv preprint arXiv:2010.10693, 2020] by using the rotation operator in three dimensional sphere, we obtain an exponential decay estimate for the diameter of agents' positions as well as time-asymptotic flocking for a class of initial data. The sufficient condition for the initial data depends only on the communication rate and inter-particle bonding parameter but not the number of agents. The lack of momentum conservation and the curved space domain make it difficult to apply the standard methodology used in the original Cucker-Smale model. To overcome this and obtain a uniform position alignment estimate, we use an energy dissipation property of this system and transform the Cucker-Smale type flocking model into an inhomogeneous system of differential equations of which solution contains the position  and velocity diameters. The coefficients of the transformed system are controlled by the communication rate and a uniform upper bound of  velocities  obtained by the energy dissipation.
\end{abstract}

\maketitle


%
%
\section{Introduction}
\setcounter{equation}{0}
Many species in nature such as birds, fish, and small germs form cluster to survive. Researchers have conducted various studies to understand this clustering phenomenon for the past several decades \cite{Adler,  Kuromoto, T-T, Winfree}. The flexility for real world applications is one of the major  reasons why this phenomenon attracts attention from the many researchers. For example,  the effective control of a large number of unmanned drones by imitating nature is one of popular topics in the engineering community \cite{NSMD, chandler, CLV}. For the development of a surveillance  system, a flocking algorithm is used to cover large areas with limited resources and to track targets \cite{S-B}. This flocking phenomenon also has been widely used in various fields and has been studied intensively in the last decade. In particular, the Cucker-Smale (C-S) model is one of important models that sparked various types of mathematical researches  in this field.

Cucker and Smale \cite{CucSma07} introduced a system of ordinary differential equations (ODEs) given by
\begin{align*}
\dot{x}_i=v_i,\qquad 
\dot{v}_i=\sum_{j=1}^{N} \psi_{ij}(v_j - v_i),
\end{align*}
where $x_i$ and $v_i$ are the position and the velocity of the $i$th agent for $\idn$, respectively. Moreover, $\psi_{ij}$ is the communication rate between $i$th and $j$th agents. We notice that  the C-S model contains the acceleration term described by weighted internal relaxation forces.

In this paper,  we focus on the complete position alignment of the corresponding C-S type flocking model on a sphere when it contains an inter-particle bonding force. We present a new framework to obtain the complete position flocking under a sufficient condition of initial data.  We emphasize that our  condition for the complete position flocking is independent of the number of agents. In particular, we prove that this second order dynamical system has a uniform exponential decay rate for the diameter of agents' positions.

From the nature of the model on sphere, the avoidance of antipodal points is necessary to guarantee the formation of a group as in Definition~\ref{def:flo}. However, due to the curved geometry, it is hard to control the position diameter of ensemble $\{(x_i,v_i)\}_{i=1}^N$, even assuming sufficiently fast velocity alignment.  Thus, motivated by the flat space case studied in \cite{PKH10},  we have added a modified inter-particle bonding force to the flocking model on sphere  to control the position diameter in our previous paper \cite{C-K-S}.  Due to the geometric property of sphere, a direct application of the bonding force term in $\bbr^3$ to the model on a sphere may disrupt the motion on the sphere. Instead, we employ a coupling force based on the Lohe operator in \cite{CCH14, Lohe}
and derive the following flocking model \cite{C-K-S} on a sphere with inter-particle bonding forces:
\begin{subequations}
\label{maino}
\begin{align}
\label{mainoa}
\dot{x}_i&=v_i,\\
\label{mainob}
\dot{v}_i&= -\frac{\|v_i\|^2}{\|x_i\|^2}x_i +\sum_{k=1}^N\frac{\psi_{ik}}{N}(R_{x_k \rightarrow x_i} (v_k)- v_i) + \sum_{k=1}^N \frac{\sigma }{N}(x_k - \langle x_i,x_k \rangle  x_i),
\end{align}
\end{subequations}
where $\psi_{ij}$ is the communication rate between $i$th and $j$th agents and   $R_{\cdot\rightarrow \cdot}$ is a rotation operator given by
\[R_{x_k \rightarrow x_i} (v_j)=R(x_k,x_i)\cdot v_j\]
and for $x_k\ne x_i$,
\begin{align}
\label{eqn:rot}
\begin{aligned}
R(x_k,x_i):=
\langle x_k,x_i\rangle  I + x_i x_k^T - x_k x_i^T + (1-  \langle x_k,x_i\rangle) \left( \frac{x_k \times x_i}{|x_k \times x_i|} \right) \left( \frac{x_k \times x_i}{|x_k \times x_i|} \right)^T.
\end{aligned}
\end{align}
Here, $x_k$, $x_i$ and $v_j$ are three dimensional column vectors. We will discuss the properties of the rotational operator in detail in the next section.

The third term in the right hand side of \eqref{mainob} is one of the cooperative control laws and $\sigma>0$ is the inter-particle bonding force rate. We note that \eqref{mainob} only contains the attractive force. In general, the cooperative control  law consists of the sum of attractive and repulsive forces. See \cite{C-K-S2}.  The research on the cooperative control  law of multi-agent systems such as \eqref{maino}  is steadily increasing \cite{M,R-B}  in the engineering field after the development of wireless communication technology. The flocking, agreement, formation and collision avoidance are their main subjects \cite{L-S, M-K-K,O,S-B}. For example, in  \cite{D}, the opinions of committee are regarded as points and the conditions for convergence are provided.  The authors in \cite{A-N} proposed a controller that yields the angular position synchronization of robot systems. Recently, for practical reasons, the research has been conducted in several restricted cases  such as  system without velocity information \cite{A-T,A-N}, limited visibility robots \cite{A-O-S-Y,F-P-S-W} and  objects on a sphere \cite{L-Z,L-S}.  In these studies, consensus algorithms was used  to allow the individuals in the system to behave as one group.  The corresponding natural  rendezvous concept is  given by
\begin{definition}\label{def 1.2} \cite{L-S} Let $\{(x_i,v_i)\}_{i=1}^N $ be the solution  to \eqref{maino}.  The ensemble $\{(x_i,v_i)\}_{i=1}^N $ has an asymptotic rendezvous, if
\[\lim_{t\to\infty }\max_{i,j}\|x_i(t)-x_j(t)\|=0.\]
\end{definition}

By using the rotation operator $\ro$, we can also define the flocking on a sphere.

\begin{definition}
\label{def:flo}
\cite{C-K-S}
A dynamical system  on a sphere has time-asymptotic flocking if its solution $\{(x_i,v_i)\}_{i=1}^N $ satisfies the following conditions:
\begin{itemize}
\item (velocity alignment) the relative velocity of any two agents goes to zero  as time goes to $\infty$:
\begin{align}\nonumber
\lim_{t \rightarrow \infty} \max_{1\leq i,j\leq N} \|x_i(t) + x_j(t)\| \|  \rji v_j(t) - v_i(t) \| = 0.
\end{align}
\item (antipodal points avoidance) any two agents are not located at the antipodal points for all $t\geq 0$:
\[\liminf_{t\geq 0}\min_{1\leq i,j\leq N}\|x_i(t)+x_j(t)\|>0.\]
\end{itemize}
\end{definition}

In \cite{C-K-S}, we obtain that the flocking model has the velocity alignment property for any $\sigma\geq 0$.
The model has the time-asymptotic flocking for a given $\sigma>0$ with the initial data satisfying a sufficient condition depending on $\psi$, $\sigma$, and $N$. See Theorem \ref{thm:0} in Section 2. The purpose of this paper is to remove the dependence of $N$ and obtain an exponential decay rate. The following energy functional, motivated by \cite{PKH10}, plays a crucial role in the proof of the main theorem in this paper as well as  \cite{C-K-S}.
\begin{definition}
 For a solution $\xvi$ to \eqref{maino},  the energy functional $\E(t)=\E(x(t),v(t))$ is defined by
\begin{align}
\label{eqn:e}
\E:= \E_K + \E_C, \quad \E_K(t):= \frac{1}{N}\sum_{k=1}^N\|v_k(t)\|^2, \quad  \E_C(t) := \frac{\sigma}{2N^2 } \sum_{k,l=1}^N \| x_k(t) - x_l(t) \|^2.
\end{align}
\end{definition}
 We note that in \cite{C-K-S},  to obtain the antipodal avoidance or the velocity alignment, the main difficulty comes from the last term of the operator $\rot$ given in \eqref{eqn:rot}:
 \begin{align}\label{R3}
(1-  \langle z_1,z_2\rangle ) \left( \frac{z_1 \times z_2}{|z_1 \times z_2|} \right) \left( \frac{z_1 \times z_2}{|z_1 \times z_2|} \right)^T. \end{align}
Due to this term, $dR/dt$ can be singular when $x_i(t) = -x_j(t)$ for some $t>0$. This antipodal configuration corresponds to $x_i=\infty$ in the original C-S case.  Even assuming an exponentially fast flocking, we cannot control the position diameter due to geometric constraints. In \cite{C-K-S}, to avoid this singularity and obtain the flocking theorem for the inter-particle bonding force, we first construct energy inequality:
\begin{align*}
\E(t) +  \sumij \int_0^t \frac{\psi_{ij}}{N^2} \| R_{x_j(s) \rightarrow x_i(s)}(v_j(s))-v_i(s)\|^2 ds \leq \E(0) \quad\hbox{ for all } t \in \Rpz.
\end{align*}
This energy inequality yields a uniform positive lower bound of $\| x_i + x_j \|$ under a sufficient condition of initial data depending  on the number of agents $N$. For more details, see \cite{C-K-S}. From this bound, we  showed uniform Lipschitzness of $v_i$ as well as $\rji v_j$ and we concluded that the asymptotic flocking occurs  by Barbalat's lemma.

However, with reference to  the flat space case, it is a natural expectation that asymptotic exponential rendezvous will happen and the  sufficient condition of initial data is independent of the number of agents $N$. To obtain the uniform position alignment for the spherical model in \eqref{maino}, we crucially use the boundedness of the energy $\E$ and the modulus conservation property of $R$. Unlike the flocking result in \cite{C-K-S}, in the complete position alignment point of view, $ x_i x_k^T$ and $- x_k x_i^T $ terms in the rotation operator  $\ro$ causes the main difficulty, but the modulus conservation property of  $\ro$  via \eqref{R3} enables us to prove our main result.

Throughout this paper, we assume that the communication rate $\psi_{ij}$ satisfies
\begin{itemize}
\item[$(\mathcal{H}1)$]  $\psi_{ij}=\psi(\|x_i-x_j\|)$,
\item[$(\mathcal{H}2)$] $\psi$ is a nonnegative strictly decreasing function with $\psi(2)=0$,
\item[$(\mathcal{H}3)$] $\psi$ is $C^1$ function on $[0,2]$.
\end{itemize}
With this assumptions on $\psi$, we obtain the following complete position flocking result.
\begin{maintheorem}
\label{thm:1}
Assume that $\psi_{ij}$ satisfies $(\mathcal{H}1)$-$(\mathcal{H}3)$ and the initial data satisfy that
\begin{align*}
\max_{1\leq k\leq N}\|v_k(0)\|< \mathcal{V}^0,\quad   \E(0)<\E^0,
\end{align*}
\begin{align}\label{condition for DxDv}
\max_{1\leq i,j\leq N}\|x_i(0)-x_j(0)\|< \mathcal{D}_x^0,\quad\max_{1\leq i,j\leq N}\|v_i(0)-v_j(0)\|<\mathcal{D}_v^0.
\end{align}
Then the solution to \eqref{maino} has time-asymptotic flocking  on a unit sphere and  exponential rendezvous
\[\max_{1\leq i,j\leq N}\|x_i(t)-x_j(t)\|\leq \max_{1\leq i,j\leq N}\|x_i(0)-x_j(0)\|e^{-\delta t }, \]
where $\delta$, $\mathcal{V}^0$, $\E^0$, $\mathcal{D}_x^0$, and  $\mathcal{D}_v^0$ are positive constants depending on $\psi$ and $\sigma$ only.

\end{maintheorem}
\begin{remark}
$\,$
\begin{enumerate}
\item
The global-in-time existence and uniqueness of the solution to \eqref{maino} is proved in \cite{C-K-S}. 
  \item In \cite{C-K-S2}, we obtain that there is  $\E^0>0$ such that if
  \begin{align}\label{condi_E}
    \E(0)<\E^0,
  \end{align} then the ensemble $\{(x_i,v_i)\}_{i=1}^N $ has an asymptotic rendezvous. Combining  this result and Theorem \ref{thm:1}, we can remove the condition in \eqref{condition for DxDv} and   obtain the exponential convergence result.
      The condition in \eqref{condi_E} is essential since the ensemble satisfying $\frac{1}{N}\sum_{i=1}^Nx_i(0)=0$ and $v_i(0)=0$ has not an asymptotic rendezvous.
  \item From the numerical simulations in Section 5, we can observe that  the convergence rate in Theorem \ref{thm:1} is almost optimal.  See Figure \ref{fig2}.\end{enumerate}

\end{remark}

The rest of this paper is organized as follows. In Section \ref{sec:2}, we review the definition of the flocking on the sphere and provide the derivation of the C-S type model with the inter-particle bonding forces and the properties of rotation operator $\rji$. In Section~\ref{sec:3}, we provide a reduction from \eqref{maino} to an inhomogeneous system of differential equations.   In Section~\ref{sec:4}, we present the proof of the asymptotic convergence result in  Theorem \ref{thm:1} for the system with the inter-particle bonding forces. In Section~\ref{sec:5}, we use numerical simulations to confirm that our analytic results are almost optimal.   Finally, Section \ref{sec:6} is devoted to the summary of our main results.\\

{\bf Notation:}
After normalization, we consider that the domain is a unit sphere $\D$ defined by
\begin{align*}
\mathcal{D} :=\{(a,b,c)\in \bbr^3: a^2+b^2+c^2=1 \}
\end{align*}
and we set
\[ x := (x_1, \ldots, x_N) \in \D^N, \quad v := (v_1, \ldots, v_N) \in \bbr^{3N}. \]
For a given $z_1,z_2 \in \bbr^3$, we use  $\langle z_1,z_2\rangle$ to denote the standard inner product in $\bbr^3$ and the standard symbol
 \[\|z_1\| = \|z_1\|_2=\sqrt{\langle z_1,z_1\rangle}\]
 to denote the $\ell_2$-norm.

\section{Flocking model with Lagrange multiplier and inter-particle bonding forces}
\label{sec:2}
\setcounter{equation}{0}
In this section, we  review the definition of the flocking on the sphere  in general geometrical setting and the results for the rotation operator $\rji$ and the energy functional $\E$ from  \cite{C-K-S}. The properties of the rotation operator and the energy functional are essential ingredient for the proof of our main theorem. \subsection{Relative velocity on a sphere} Unlike the flat space $\bbr^3$, in a general manifold, if two agents have different positions, the corresponding velocities are belonged to different tangent spaces, respectively. Therefore, to give the meaning that velocities in different tangent spaces are aligned, it is necessary to define a kind of transformation  between two different tangent spaces. Thus, we  define the following velocity difference $D_v$ between $v_i$ and $v_j$ at $x_i$  in the most geometrically canonical way:
\[  \rji (v_j(t)) - v_i(t). \]
We note that defining relative velocity is a topic that has received a lot of attention in general manifold theory. In particular, it was covered in depth in  general relativity \cite{Car97, McG03,Sch85, Tal08}. The basic idea is largely similar,  it is a parallel transport along geodesics on a manifold to compare sizes in a tangent space at one position of the manifold \cite{Car97,Spi79b}. From this general observation, a parallel transport on a sphere is characterized by a rotation matrix given in \eqref{eqn:rot}.

The central idea of a relative velocity is to consider the geodesic  for two given points, which is the shortest path between two points. Then, we transport a vector field in a tangent space at one point to the tangent space at another point along the geodesic. Let $M$ be a $n$-dimensional manifold. Note that if $\M = \Rn$, then $T_x \Rn = \Rn$ under the natural identification. For $x \in \M$, the tangent space  $T_x \M$ of $\M$ at $x$ is defined as the set of all tangent vectors of $\M$ at $x$. We say that a vector field $V$ along a curve $\gam : [a,b] \to \M$ is said to be parallel along $\gam$ if $D_t V = 0$ in $[a,b]$. Here, $D_t V$ is the covariant derivative along $\gam$ obtained by the normal projection of $dV_{\gam(t)}/dt$ onto the tangential plane of $\D$ at $\gam(t)$. Note that geodesics in $\R^d$ are straight lines and a constant vector is parallel along a straight line. See  \cite{DoC16, Lee06, Spi79} for the general reference. We recall the existence and uniqueness of the parallel transport along a curve from \cite{Lee06}.
\begin{lemma} \cite[Theorem 4.11]{Lee06}
Given a curve $\gam : [0,t_1] \rightarrow \M$ and a vector $W_0 \in T_{\gam(0)} \M$, there exists a unique parallel vector field $V$ along $\gam$ such that $V_{\gam(0)} = W_0$.
\end{lemma}
Note that a geodesic between two points $z_1$ and $z_2$ on a sphere is a part of a great circle containing $z_1$ and $z_2$. Also, it is uniquely determined unless $z_1 = -z_2$. It is well-known that the parallel transport of a vector $W_0$ along a great circle is given by $RW_0$ for some rotation matrix $R$. See Chapter 4.4 in \cite{DoC16}. We prove the following proposition for the sake of completeness.

\begin{proposition}
Let $\gam : [0,t_1] \rightarrow \D$ be a geodesic on a sphere and $W_0 \in T_{\gam(0)} \D$. If
\begin{align}
\label{eqn:1geo}
\gam(0) + \gam(t) \neq 0
\end{align}
for all $t \in [0,t_1]$, then a vector field defined by
\begin{align}
\label{eqn:2geo}
V_{\gam(t)} := R_{\gam(0) \rightarrow \gam(t)} W_0
\end{align}
is parallel along $\gam$, where $R$ is the rotation operator given in \eqref{eqn:rot}.
\end{proposition}

\begin{proof}
By the symmetry of a sphere and the condition in \eqref{eqn:1geo}, it is enough to consider a geodesic $\gam : [0,\theta_1] \rightarrow \D$ from $e_1$ to $z_1$ for some $\theta_1 \in (0,\pi)$ given by
\begin{align*}
\gam(t) = (\cos(t), \sin(t), 0),
\end{align*}
where $e_1 := (1,0,0)$, $ e_2:=(0,1,0)$, $e_3 := (0,0,1)$, and $z_1 := (\cos(\theta_1), \sin(\theta_1),0)$.

  We show that $R_{\gam(0) \rightarrow \gam(t)} e_i$ is parallel along $\gam$ for all $t \in [0,\theta_1]$ and $i \in \{2,3\}$.
By the direct computation, it holds that
\begin{align}
\label{eqn:2ex1}
\rji =
\begin{bmatrix}
\cos(\alpha_i - \alpha_j) & -\sin(\alpha_i - \alpha_j) &0  \\
\sin(\alpha_i - \alpha_j) & \cos(\alpha_i - \alpha_j)  &0 \\
0 & 0 & 1
\end{bmatrix}.
\end{align}
From \eqref{eqn:1geo}, we can use \eqref{eqn:2ex1} to obtain
\begin{align*}
R_{\gam(0) \rightarrow \gam(t)} e_2 = (-\sin(t), \cos(t), 0) \quad\hbox{ and }\quad R_{\gam(0) \rightarrow \gam(t)} e_3 = e_3.
\end{align*}

Thus, it follows that
\begin{align*}
\frac{d R_{\gam(0) \rightarrow \gam(t)} e_2}{dt} = - \gam(t) \quad \hbox{ and }\quad \frac{d R_{\gam(0) \rightarrow \gam(t)} e_3}{dt} = 0.
\end{align*}
As the covariant derivatives are tangential components of above equations, $D_t R_{\gam(0) \rightarrow \gam(t)} e_i$ is zero for all $t \in [0,\theta_1]$ and $i \in \{2,3\}$. Note that $\{e_2, e_3\}$ is a basis of $T_{e_1}\D$. Therefore, $W_0$ can be written as a linear combination of $e_2$ and $e_3$ and we conclude that a vector field given in \eqref{eqn:2geo} is parallel along $\gam$.
\end{proof}

We emphasize that the rotation operator $\ro$ is an isometry as well as a bijection between two tangent spaces. We also present properties of the operator $\ro$ for use in the next sections.
\begin{lemma}
\label{lem:rot}
For $z_1, z_2 \in \D$ such that $z_1 \neq z_2$ , $z_1 \neq -z_2$ and $\rot$ given in \eqref{eqn:rot}, it holds that
\begin{align}
\label{eqn:1rot}
R_{z_1\rightarrow z_2}(z_1) = z_2, \quad R_{z_1\rightarrow z_2}(z_2) = 2\langle z_1, z_2\rangle  z_2 - z_1, ~ \hbox{ and }~  R_{z_1\rightarrow z_2}(z_1 \times z_2) = z_1 \times z_2.
\end{align}
Furthermore, we have
\begin{align}
\label{eqn:2rot}
R_{z_1\rightarrow z_2}^{T} = R_{z_2\rightarrow z_1}. 
\end{align}
In particular, $\rot$ is an orthogonal matrix, that is
\begin{align}
\label{eqn:3rot}
\rott \rot = I.
\end{align}
\end{lemma}

\begin{proof}
For the proof of this lemma, see Lemma 2.4 in \cite{C-K-S}.
%
%
%
%
\end{proof}

\begin{proposition}
\label{prop:tan}
$R_{z_1\rightarrow z_2} |_{T_{z_1} \D}$ is a bijection and an isometry from $T_{z_1} \D$ to $T_{z_2} \D$.
\end{proposition}

\begin{proof}
From \eqref{eqn:1rot} and \eqref{eqn:2rot} in Lemma~\ref{lem:rot}, it holds that for any $v \in \R^3$,
\begin{align}
\label{eqn:tan11}
\langle \rot (v), z_2\rangle  = v^T \rott z_2 = v^T \rotr z_2 = v^T z_1=\langle v, z_1\rangle.
\end{align}
As $\D$ is a unit sphere, we have
\begin{align}
\label{eqn:tan12}
\langle v,z_1\rangle  = 0 \quad \hbox{ for any vector } v \in T_{z_1} \D.
\end{align}
From \eqref{eqn:tan11} and \eqref{eqn:tan12}, we conclude that  $\langle \rot (v),z_2\rangle = 0$ for any vector $v \in T_{z_1} \D$ and thus
\[\rot (v) \in T_{z_2} \D.\]

Furthermore, as $\rot$ is an orthogonal matrix from Lemma~\ref{lem:rot}, it is invertible. By \eqref{eqn:2rot} and \eqref{eqn:3rot} in Lemma~\ref{lem:rot}, it holds that for any $v \in \R^3$,
\begin{align*}
\langle \roti (v),z_1\rangle  =\langle  v,z_2\rangle.
\end{align*}
From \eqref{eqn:tan11}, $\roti (v) \in T_{z_1} \D$ for any vector $v \in T_{z_2} \D$ and we conclude that $\rot$ is a bijection.

Lastly, as $\rot$ is an orthogonal matrix from Lemma~\ref{lem:rot}, it is an isometry.
\end{proof}

\subsection{Lagrange multiplier and energy dissipation on a sphere}
The C-S type flocking model on a sphere in \eqref{maino} has the following property of the velocity alignment:
\begin{align*}
\lim_{t \rightarrow \infty} \max_{1\leq i,j\leq N} \|x_i+x_j\|\| \rji (v_j(t)) - v_i(t) \| = 0.
\end{align*}
Here, $\|x_i+x_j\|$ term in the above flocking limit naturally appear from the geometric structure of the sphere. See \cite{C-K-S} for the detailed argument. Notice that if we assume that  there is a constant $C_l>0$ such that $\|x_i+x_j\|\geq C_l$  for any $t>0$ and $i,j\in \bbn$, then the above limit is equivalent to
\begin{align*}
\lim_{t \rightarrow \infty} \max_{1\leq i,j\leq N} \| \rji (v_j(t)) - v_i(t) \| = 0
\end{align*}
as  the form of the velocity alignment  in the flat space.

If the bonding force rate $\sigma$ is large enough comparing the differences of agents' velocities and positions, then the following flocking result with position alignment  holds.

\begin{theorem}\cite{C-K-S}
\label{thm:0}
Assume that $\psi_{ij}$ satisfies $(\mathcal{H}1)$-$(\mathcal{H}3)$. If $2\sigma > N^2\E(0)$, then the solution to \eqref{maino} has time-asymptotic flocking  on a unit sphere. Here, $\E(0)$ is the initial energy of the system given in \eqref{eqn:e}.
\end{theorem}

 We revisit the idea for deriving the flocking model introduced in \cite{C-K-S}.  Then, the centripetal force term of the flocking model will be explained using the Lagrange Multiplier, and the inter-particle bonding force term corresponding to the sphere will be defined.  For the consistency of initial data on the unit sphere, we consider  following initial conditions:
\begin{align}
\label{eqn:ini}
\langle v_i(0),x_i(0)\rangle=0\quad \hbox{ and }\quad \|x_i(0)\|= 1, \quad \mbox{for all}~ 1\leq i\leq N.
\end{align}

The augmented C-S model in the Euclidean space (See \cite{PKH10}), the following term is added as inter-particle bonding forces:
 \[\displaystyle \frac{\sigma}{N} \sum_{k=1}^N (x_k - x_i).\]
Here, $\sigma > 0$ is the rate of the inter-particle bonding force. However, this term will prevent that the agent is located in the sphere. Thus, we adapt a modified inter-particle bonding force in \cite{C-K-S}.  We notice that the modified inter-particle bonding forces will be the form of Lohe operator in \cite{Lohe}. In summary,  we consider the following model with the Lagrange multiplier $\lambda_i x_i$ for the controllability of position difference.
\begin{subequations}
\label{maini}
\begin{align}
\label{mainia}
\dot{x}_i&=v_i,\\
\label{mainib}
\dot{v}_i&= \lam_i x_i +\sum_{k=1}^N\frac{\psi_{ik}}{N}(R_{x_k \rightarrow x_i} (v_k)- v_i)+ \frac{\sigma}{N} \sum_{k=1}^N (x_k -  x_i),
\end{align}
\end{subequations}
where $R_{x_k \rightarrow x_i} $ is an operator from $T_{x_j} \D$ to $T_{x_i} \D$.

Based on the idea of \cite[Proposition 2.2]{C-K-S}, we choose $\{\lam_i\}_{\idn}$ for the flocking model with the inter-particle bonding force as follows:
\begin{align}
\label{eqn:lam}
\lam_i = - \frac{\| v_i \|^2}{\| x_i \|^2} - \frac{\sigma}{N} \sum_{k=1}^N \frac{\langle x_k - x_i, x_i \rangle}{\langle x_i,x_i\rangle}.
\end{align}
Then if initial data satisfy \eqref{eqn:ini}, then we will show that all agents are located in the unit sphere for all time. We first show that $x_i$ is on the unit sphere and $v_i$ is in the tangent space of $\D$ at $x_i$ for all $\idn$.

\begin{proposition}
\label{prop:uni}
For $\{\lam_i\}_{\idn}$ given in \eqref{eqn:lam} and $t\in \Rpz$, the solution to \eqref{mainia}-\eqref{mainib} subject to \eqref{eqn:ini} satisfies that
\begin{align}
\label{eqn:1uni}
\langle v_i(t),x_i(t)\rangle=0, \quad \|x_i(t)\|= 1, \quad \mbox{for all}~ 1\leq i\leq N .
\end{align}
\end{proposition}
\begin{proof}
We claim that
\begin{align}
\label{eqn:uni11}
\frac{d}{dt}\langle v_i,x_i\rangle = 0.
\end{align}
From \eqref{mainia}-\eqref{mainib}, it follows that
\begin{align*}
\frac{d}{dt}\langle v_i,x_i\rangle &=\langle \dot{v}_i,x_i\rangle+ \|v_i\|^2\\& = \lam_i \| x_i \|^2+ \sum_{k=1}^N\frac{\psi_{ik}}{N}\langle R_{x_k \rightarrow x_i} v_k- v_i, x_i \rangle + \frac{\sigma}{N} \sum_{k=1}^N \langle x_k - x_i, x_i \rangle + \|v_i\|^2.
\end{align*}
As $R_{x_k \rightarrow x_i} v_k \in T_{x_i} \D$ and $\lam_i$ is given in \eqref{eqn:lam},
\begin{align*}
\frac{d}{dt}\langle v_i,x_i\rangle = - \bigg( \sum_{k=1}^N\frac{\psi_{ik}}{N} \bigg) \langle v_i, x_i \rangle,
\end{align*}
and for all $t \in \Rpz$,
\begin{align*}
\langle v_i(t),x_i(t)\rangle = \langle v_i(0),x_i(0)\rangle \exp \bigg(  \int_0^t \frac{1}{N}\sum_{k=1}^N\frac{\psi_{ik}(s)}{N} ds \bigg) = 0.
\end{align*}

On the other hand, by \eqref{eqn:uni11}, we have
\begin{align*}
\frac{d}{dt} \|x_i \|^2 = 2 \langle v_i,x_i\rangle = 0
\end{align*}
and thus we conclude \eqref{eqn:1uni}.
\end{proof}

Lastly, recall the following energy dissipation property. As mentioned before, this dissipation plays an important role in the proof of the flocking in \cite{C-K-S}. We also crucially use this property when we  prove the complete position flocking behavior.
\begin{proposition}
\label{prop:vt}\cite{C-K-S} Let $\xvi$ be the solution to \eqref{maino} with $(\mathcal{H}1)$-$(\mathcal{H}3)$.
Then, the following holds for all $t \in \Rpz$,
\begin{align}
\label{eqn:vt}
\frac{d\E}{dt} = - \sumij  \frac{\psi_{ij}}{N^2}\| R_{x_j\rightarrow x_i}(v_j)-v_i\|^2,
\end{align}
where $\E$ is the energy function of $\xvi$ defined in \eqref{eqn:e}.
As a consequence, we have
\begin{align}
\label{eqn:2vt}
\E(t) +  \sumij \int_0^t \frac{\psi_{ij}}{N^2} \| R_{x_j(s) \rightarrow x_i(s)}(v_j(s))-v_i(s)\|^2 ds \leq \E(0) \quad\hbox{ for all } t \in \Rpz.
\end{align}
\end{proposition}

\section{Reduction to a linearized system of equations with a negative definite coefficient matrix}
\label{sec:3}
\setcounter{equation}{0}
In this section, we derive a linearized system of equations from the C-S type flocking model in \eqref{maino}. As mentioned before, the main obstacle in proving our main result comes from the lack of a conserved quantity. Compared to the original C-S model, the flocking model on sphere has no momentum conservation. Therefore, we cannot use the standard methodology using in the C-S model. On the other hand, the linearized system \eqref{eq:X} with a negative definite coefficient matrix gives new sharp estimates on the diameters of positions and velocities. This leads the complete position flocking result in Section~\ref{sec:4}. Additionally, we notice that our uniform estimates does not depend on the number of agents $N$. In order to obtain a uniform analysis regardless of $N$, we need a global upper bound of physical quantities below, not the upper bound of their average as in \cite{C-K-S}.

For given $i,j\in \{1,\ldots,N\}$ and $\xvi$, consider the vector-valued functional $X^{ij}(t)$ given by
\begin{align}\label{def:X}X^{ij}(t):=(X^{ij}_1(t),X^{ij}_2(t),X^{ij}_3(t))^T,\end{align} where
\begin{align}\label{def:X12}X^{ij}_1(t):=\|x_i(t)-x_j(t)\|^2,\qquad X^{ij}_2(t):= \langle v_i(t)-v_j(t),x_i(t)-x_j(t)\rangle,\end{align}
and
\begin{align}\label{def:X3}  X^{ij}_3(t):=\| v_i(t)-v_j(t)\|^2.\end{align}

In Proposition~\ref{prop 4.2}, we prove that $X^{ij}$ satisfies the system of linear differential equations in \eqref{eq:X}, which has the following inhomogeneous terms,
\begin{align}
\label{def:F}
F^{ij}(t):=(F^{ij}_1(t),F^{ij}_2(t),F^{ij}_3(t))^T,
\end{align}
where $F^{ij}_1$, $F^{ij}_2$, and  $F^{ij}_3$ are defined by
 \begin{align}\begin{aligned}
 \label{def:F12}
  F^{ij}_1(t):=&0,\\ F^{ij}_2(t):=&
-\frac{\|v_i\|^2+\|v_j\|^2}{2}\|x_i-x_j\|^2+\frac{\psi_0}{N}\sum_{k=1}^N
\langle R_{x_k \rightarrow x_i} (v_k)
-R_{x_k \rightarrow x_j} (v_k)
,x_i-x_j\rangle
\\
&+
\sum_{k=1}^N\left (\frac{\psi_{ik}}{N}-\frac{\psi_{ii}}{N}\right)\langle R_{x_k \rightarrow x_i} (v_k)- v_i
,x_i-x_j\rangle\\&
-
\sum_{k=1}^N\left (\frac{\psi_{jk}}{N}-\frac{\psi_{jj}}{N}\right)\langle R_{x_k \rightarrow x_j} (v_k)- v_j
,x_i-x_j\rangle
\\&
 +\frac{\sigma}{4N}\sum_{k=1}^N\|x_k-x_i\|^2 \|x_i-x_j\|^2+\frac{\sigma}{4N}\sum_{k=1}^N\|x_k-x_j\|^2 \|x_i-x_j\|^2
 ,
 \end{aligned}\end{align}
and
 \begin{align} \begin{aligned}\label{def:F3}
 F^{ij}_3(t):=&2\langle -\|v_i\|^2x_i
 +\|v_j\|^2x_j
 ,v_i-v_j\rangle
\\
&
+
2\sum_{k=1}^N\frac{\psi_0}{N} \left\langle R_{x_k \rightarrow x_i} (v_k)
-R_{x_k \rightarrow x_j} (v_k)
,v_i-v_j\right\rangle\\&
 +2\sum_{k=1}^N\left(\frac{\psi_{ik}}{N}-\frac{\psi_{ii}}{N}\right)\left\langle R_{x_k \rightarrow x_i} (v_k)- v_i
,v_i-v_j\right\rangle\\
&
-2\sum_{k=1}^N\left(\frac{\psi_{jk}}{N}-\frac{\psi_{jj}}{N}\right)\left\langle R_{x_k \rightarrow x_j} (v_k)- v_j
,v_i-v_j\right\rangle
\\
&
+
\frac{2\sigma}{N}\sum_{k=1}^N  ( \langle x_i,x_k \rangle-1) \langle  x_i,v_j\rangle
+\frac{2\sigma}{N}\sum_{k=1}^N
 (\langle x_j,x_k \rangle-1) \langle  x_j
,v_i\rangle. \end{aligned}
 \end{align}

\begin{proposition}\label{prop 4.2}
Let $\xvi$ be the solution to \eqref{maino}. For any $i,j\in \{1,\ldots,N\}$, the vector-valued functional $X^{ij}$ defined in \eqref{def:X}-\eqref{def:X3} satisfies \phantom{\eqref{def:F12}\eqref{def:X12}}
 \begin{align}\label{eq:X}
\frac{d}{dt}X^{ij}(t)=A X^{ij}+F^{ij},
 \end{align}
where $F^{ij}$ is the functional defined in \eqref{def:F}-\eqref{def:F3} and for  positive constants $\sigma $ and $\psi_0:=\psi(0)$, the coefficient matrix $A$ is given by
 \begin{align*}
A =
\left(\begin{matrix}
0 & 2 &0  \\
-\sigma & -\psi_0  &1 \\
0 & -2\sigma & -2\psi_0
\end{matrix}\right).
\end{align*}

\end{proposition}

\begin{remark}
We will verify that the leading coefficients of the linearized system is the sum of a negative definite matrix and controllable quantities by the energy in \eqref{eqn:e}.
\end{remark}

\begin{proof}[Proof of Proposition~\ref{prop 4.2}]
  From direct calculation, it follows that
\begin{align}\label{pf:x1}
\frac{d}{dt}X^{ij}_1=\frac{d}{dt} \|x_i-x_j\|^2= 2\langle v_i-v_j,x_i-x_j\rangle=2X^{ij}_2,
\end{align}
and
\begin{align*}
\frac{d}{dt}X^{ij}_2&=\frac{d}{dt}  \langle v_i-v_j,x_i-x_j\rangle
=\langle v_i-v_j,v_i-v_j\rangle+\langle \dot v_i-\dot v_j,x_i-x_j\rangle
=X^{ij}_3+\langle \dot v_i-\dot v_j,x_i-x_j\rangle.
\end{align*}

By \eqref{mainob}, we obtain that
\begin{align*}
\langle \dot v_i-\dot v_j,x_i-x_j\rangle&=
\langle -\|v_i\|^2x_i
 +\|v_j\|^2x_j
 ,x_i-x_j\rangle
\\&\quad  +\left\langle\sum_{k=1}^N\frac{\psi_{ik}}{N}(R_{x_k \rightarrow x_i} (v_k)- v_i)
-\sum_{k=1}^N\frac{\psi_{jk}}{N}(R_{x_k \rightarrow x_j} (v_k)- v_j)
,x_i-x_j\right\rangle
\\
&\quad+\left\langle \frac{\sigma}{N} \sum_{k=1}^N (x_k - \langle x_i,x_k \rangle  x_i)- \frac{\sigma}{N} \sum_{k=1}^N (x_k - \langle x_j,x_k \rangle  x_j)
,x_i-x_j\right\rangle
\\&:=K^{ij}_1+K^{ij}_2+K^{ij}_3.
\end{align*}

For $K^{ij}_1$, we use the conservation property of $\|x_i\|$.
\begin{align*}
K^{ij}_1=
 -\|v_i\|^2+ \|v_i\|^2\langle x_i,x_j\rangle
  -\|v_j\|^2+ \|v_j\|^2\langle x_i,x_j\rangle=-\frac{\|v_i\|^2+\|v_j\|^2}{2}\|x_i-x_j\|^2
\end{align*}
Note that  $\psi_0=\psi_{ii}$ for all $i\in \{1,\ldots,N\}$.
For $K^{ij}_2$, we have
\begin{align*}
 K^{ij}_2&=\left\langle\sum_{k=1}^N\frac{\psi_{ik}}{N}(R_{x_k \rightarrow x_i} (v_k)- v_i)
-\sum_{k=1}^N\frac{\psi_{jk}}{N}(R_{x_k \rightarrow x_j} (v_k)- v_j)
,x_i-x_j\right\rangle\\
&= \left\langle\sum_{k=1}^N\frac{\psi_0}{N} (R_{x_k \rightarrow x_i} (v_k)- v_i)
-\sum_{k=1}^N\frac{\psi_0}{N}(R_{x_k \rightarrow x_j} (v_k)- v_j)
,x_i-x_j\right\rangle\\&
\quad
+
\left \langle\sum_{k=1}^N\left(\frac{\psi_{ik}}{N}-\frac{\psi_{ii}}{N}\right)(R_{x_k \rightarrow x_i} (v_k)- v_i)
,x_i-x_j\right\rangle
\\&\quad-
\left \langle\sum_{k=1}^N\left(\frac{\psi_{jk}}{N}-\frac{\psi_{jj}}{N}\right)(R_{x_k \rightarrow x_j} (v_k)- v_j)
,x_i-x_j\right\rangle\\
&=
-\psi_0\langle v_i
- v_j
,x_i-x_j\rangle
+\frac{\psi_0}{N}\sum_{k=1}^N
\langle R_{x_k \rightarrow x_i} (v_k)
-R_{x_k \rightarrow x_j} (v_k)
,x_i-x_j\rangle
\\&
\quad
 +
\sum_{k=1}^N\left (\frac{\psi_{ik}}{N}-\frac{\psi_{ii}}{N}\right)\langle R_{x_k \rightarrow x_i} (v_k)- v_i
,x_i-x_j\rangle\\&
\quad
-
\sum_{k=1}^N\left (\frac{\psi_{jk}}{N}-\frac{\psi_{jj}}{N}\right)\langle R_{x_k \rightarrow x_j} (v_k)- v_j
,x_i-x_j\rangle
.\end{align*}
Therefore, we have
\begin{align*}
 K^{ij}_2&=
-\psi_0 X^{ij}_2
 +\frac{\psi_0}{N}\sum_{k=1}^N
\langle R_{x_k \rightarrow x_i} (v_k)
-R_{x_k \rightarrow x_j} (v_k)
,x_i-x_j\rangle
\\&
\quad
 +
\sum_{k=1}^N\left (\frac{\psi_{ik}}{N}-\frac{\psi_{ii}}{N}\right)\langle R_{x_k \rightarrow x_i} (v_k)- v_i
,x_i-x_j\rangle\\&
\quad
-
\sum_{k=1}^N\left (\frac{\psi_{jk}}{N}-\frac{\psi_{jj}}{N}\right)\langle R_{x_k \rightarrow x_j} (v_k)- v_j
,x_i-x_j\rangle.
\end{align*}

For $K_3$, we use direct calculation to obtain
\begin{align*}
  K^{ij}_3&=
 \frac{\sigma}{N}\sum_{k=1}^N\langle  - \langle x_i,x_k \rangle  x_i+\langle x_j,x_k \rangle  x_j,x_i-x_j\rangle
\\&= \frac{\sigma}{N}\sum_{k=1}^N\Big( - \langle x_i,x_k \rangle + \langle x_i,x_k \rangle\langle x_i,x_j \rangle+ \langle x_j,x_k \rangle  \langle     x_j,x_i\rangle-\langle x_j,x_k \rangle \Big)
\\&=-\frac{\sigma}{N}\sum_{k=1}^N\frac{\langle x_i,x_k \rangle+\langle x_j,x_k \rangle }{2} \| x_i-x_j\|^2
\\&=-\sigma \| x_i-x_j\|^2-\frac{\sigma}{N}\sum_{k=1}^N\frac{\langle x_i,x_k \rangle+\langle x_j,x_k \rangle-2 }{2} \| x_i-x_j\|^2
. \end{align*}

This implies that
\begin{align*}
  K^{ij}_3=-\sigma X^{ij}_1
 +\frac{\sigma}{4N}\sum_{k=1}^N\|x_k-x_i\|^2 \|x_i-x_j\|^2+\frac{\sigma}{4N}\sum_{k=1}^N\|x_k-x_j\|^2 \|x_i-x_j\|^2.
 \end{align*}
Therefore,
\begin{align*}
\frac{d}{dt}X^{ij}_2&=X^{ij}_3+K^{ij}_1+K^{ij}_2+K^{ij}_3\\
&=X^{ij}_3
-\frac{\|v_i\|^2+\|v_j\|^2}{2}\|x_i-x_j\|^2\\
&\quad
-\psi_0 X^{ij}_2
 +\frac{\psi_0}{N}\sum_{k=1}^N
\langle R_{x_k \rightarrow x_i} (v_k)
-R_{x_k \rightarrow x_j} (v_k)
,x_i-x_j\rangle
\\&
\quad
 +
\sum_{k=1}^N\left (\frac{\psi_{ik}}{N}-\frac{\psi_{ii}}{N}\right)\langle R_{x_k \rightarrow x_i} (v_k)- v_i
,x_i-x_j\rangle\\&
\quad
-
\sum_{k=1}^N\left (\frac{\psi_{jk}}{N}-\frac{\psi_{jj}}{N}\right)\langle R_{x_k \rightarrow x_j} (v_k)- v_j
,x_i-x_j\rangle
\\&
\quad-\sigma X^{ij}_1
 +\frac{\sigma}{4N}\sum_{k=1}^N\|x_k-x_i\|^2 \|x_i-x_j\|^2+\frac{\sigma}{4N}\sum_{k=1}^N\|x_k-x_j\|^2 \|x_i-x_j\|^2.
\end{align*}
Thus, we obtain the following differential equation for $X_2$.
\begin{align}\label{pf:x2}
\frac{d}{dt}X^{ij}_2=
X^{ij}_3
-\psi_0 X^{ij}_2
-\sigma X^{ij}_1+F^{ij}_2.
\end{align}

We next consider $X^{ij}_3$ case. By the definition of $ X^{ij}_3$,
\begin{align*}
\frac{1}{2}\frac{d}{dt}  X^{ij}_3=\langle \dot v_i-\dot v_j,v_i-v_j\rangle.
\end{align*}
By \eqref{mainob}, we have
\begin{align*}
\langle \dot v_i-\dot v_j,v_i-v_j\rangle&=
\langle -\|v_i\|^2x_i
 +\|v_j\|^2x_j
 ,v_i-v_j\rangle
\\&\quad  +\left\langle\sum_{k=1}^N\frac{\psi_{ik}}{N}(R_{x_k \rightarrow x_i} (v_k)- v_i)
-\sum_{k=1}^N\frac{\psi_{jk}}{N}(R_{x_k \rightarrow x_j} (v_k)- v_j)
,v_i-v_j\right\rangle
\\
&\quad+\left\langle \frac{\sigma}{N} \sum_{k=1}^N (x_k - \langle x_i,x_k \rangle  x_i)- \sigma \sum_{k=1}^N (x_k - \langle x_j,x_k \rangle  x_j)
,v_i-v_j\right\rangle\\
&:=L^{ij}_1+L^{ij}_2+L^{ij}_3.
\end{align*}
Similar to $X^{ij}_2$ case, we consider $L^{ij}_1$, $L^{ij}_2$, and $L^{ij}_3$ separately and use
\[\psi_0=\psi_{ii}\quad \mbox{ for all}\quad  i\in \{1,\ldots,N\}.\]

For $L^{ij}_2$,
\begin{align*}
L^{ij}_2&=
-\psi_0\langle v_i-
 v_j,v_i-v_j\rangle
+
\left\langle\sum_{k=1}^N\frac{\psi_0}{N}(R_{x_k \rightarrow x_i} (v_k)
-R_{x_k \rightarrow x_j} (v_k))
,v_i-v_j\right\rangle\\&
\quad +\left\langle\sum_{k=1}^N\left(\frac{\psi_{ik}}{N}-\frac{\psi_{ii}}{N}\right)(R_{x_k \rightarrow x_i} (v_k)- v_i)
-\sum_{k=1}^N\left(\frac{\psi_{jk}}{N}-\frac{\psi_{jj}}{N}\right)(R_{x_k \rightarrow x_j} (v_k)- v_j)
,v_i-v_j\right\rangle.
\end{align*}
Therefore, we have
\begin{align*}
L^{ij}_2&=
-\psi_0 X^{ij}_3
+
\sum_{k=1}^N\frac{\psi_0}{N} \left\langle R_{x_k \rightarrow x_i} (v_k)
-R_{x_k \rightarrow x_j} (v_k)
,v_i-v_j\right\rangle\\&
\quad +\sum_{k=1}^N\left(\frac{\psi_{ik}}{N}-\frac{\psi_{ii}}{N}\right)\left\langle R_{x_k \rightarrow x_i} (v_k)- v_i
,v_i-v_j\right\rangle
\\&\quad
-\sum_{k=1}^N\left(\frac{\psi_{jk}}{N}-\frac{\psi_{jj}}{N}\right)\left\langle R_{x_k \rightarrow x_j} (v_k)- v_j
,v_i-v_j\right\rangle.
\end{align*}

For $L^{ij}_3$, we have
\begin{align*}
  L^{ij}_3&=\frac{\sigma}{N}\sum_{k=1}^N\langle    - \langle x_i,x_k \rangle  x_i+ \langle x_j,x_k \rangle  x_j
,v_i-v_j\rangle\\
&=\frac{\sigma}{N}\sum_{k=1}^N  ( \langle x_i,x_k \rangle \langle  x_i,v_j\rangle + \langle x_j,x_k \rangle \langle  x_j
,v_i\rangle)
\\&=
-\sigma  \langle  x_i-x_j,v_i-v_j\rangle
+
\frac{\sigma}{N}\sum_{k=1}^N  \Big(( \langle x_i,x_k \rangle-1) \langle  x_i,v_j\rangle + (\langle x_j,x_k\rangle -1) \langle  x_j
,v_i\rangle\Big)
\\&=
-\sigma X^{ij}_2
+
\frac{\sigma}{N}\sum_{k=1}^N  \Big(( \langle x_i,x_k \rangle-1) \langle  x_i,v_j\rangle + (\langle x_j,x_k \rangle-1) \langle  x_j
,v_i\rangle\Big)
\end{align*}

Thus, we have
\begin{align*}
\frac{1}{2}\frac{d}{dt}  X^{ij}_3&=L^{ij}_1+L^{ij}_2+L^{ij}_3\\
&=
\langle -\|v_i\|^2x_i
 +\|v_j\|^2x_j
 ,v_i-v_j\rangle
-\psi_0 X^{ij}_3
+
\sum_{k=1}^N\frac{\psi_0}{N} \left\langle R_{x_k \rightarrow x_i} (v_k)
-R_{x_k \rightarrow x_j} (v_k)
,v_i-v_j\right\rangle\\&
\quad +\sum_{k=1}^N\left(\frac{\psi_{ik}}{N}-\frac{\psi_{ii}}{N}\right)\left\langle R_{x_k \rightarrow x_i} (v_k)- v_i
,v_i-v_j\right\rangle
\\&\quad
-\sum_{k=1}^N\left(\frac{\psi_{jk}}{N}-\frac{\psi_{jj}}{N}\right)\left\langle R_{x_k \rightarrow x_j} (v_k)- v_j
,v_i-v_j\right\rangle
\\
&\quad
-\sigma X^{ij}_2
+
\frac{\sigma}{N}\sum_{k=1}^N  \Big(( \langle x_i,x_k \rangle-1) \langle  x_i,v_j\rangle + (\langle x_j,x_k \rangle-1) \langle  x_j
,v_i\rangle\Big),
\end{align*}
i.e.,
\begin{align}\label{pf:x3}
\frac{d}{dt}  X^{ij}_3&=-2\psi_0 X^{ij}_3-2\sigma X^{ij}_2+F^{ij}_3.
\end{align}

By \eqref{pf:x1}, \eqref{pf:x2}, and \eqref{pf:x3}, we obtain  the differential equation in \eqref{eq:X}.
\end{proof}

\section{Uniform estimates for positions and velocities: the proof of the main theorem}
\label{sec:4}
\setcounter{equation}{0}



In this section, we complete the proof of our main theorem: the complete position alignment of the solution to \eqref{maino} when the differences of agents' initial positions and velocities and the initial maximal velocity of all agents are sufficiently small. For simplicity, we define the following Lyapunov functionals.
\[\mathcal{D}_{x}(t)=\max_{1\leq i,j\leq N}\|x_i(t)-x_j(t)\|,\qquad \mathcal{D}_{v}(t)=\max_{1\leq i,j\leq N}\|v_i(t)-v_j(t)\|.\]
 Based on the linearized system derived in Section 3, we obtain exponential decay rates for position and velocity diameters $\mathcal{D}_x(t)$, $\mathcal{D}_v(t)$ via estimating the inhomogeneous term $F^{ij}$ defined in \eqref{def:F}. The inhomogeneous term $F^{ij}$ is bounded by $c_l X^{ij}+c_h \|X^{ij}\|X^{ij}$. The higher order term  $c_h \|X^{ij}\|X^{ij}$ will be controlled by small initial data assumption and the coefficient $c_l$ of the lower order term is bounded by $\|\psi\|\mathcal{V}$, where $\mathcal{V}(t)$ is the maximal velocity defined by
\[\mathcal{V}(t)=\max_{1\leq k\leq N}\|v_k(t)\|.\]
As mentioned before, the energy functional $\E$ is decreasing. This dissipation property in Proposition \ref{prop:vt} leads to a uniform boundedness of the maximum velocity $\mathcal{V}$. Since the coefficient matrix $A$ in Proposition \ref{prop 4.2} is negative definite, combining the above properties, we can obtain the complete position flocking result.

\begin{lemma}\label{lemma 4.1}
Let $\xvi$ be the solution to \eqref{maino}. Assume that there is a constant $\psi_m>0$ such that for any $i,j \in \{1,\ldots,N\}$, $\psi_{ij}(s)\geq \psi_m$ on $0\leq s\leq t$. Then
\begin{align*}
\mathcal{V}^2(t)&\leq e^{-\frac{\psi_m}{2}t}\mathcal{V}^2(0)+(1-e^{-\frac{\psi_m}{2}t})
\bigg(  2\sup_{0\leq s\leq t }\E_K(s)+ \frac{4\sigma^2}{\psi_m^2} \sup_{0\leq s\leq t }\mathcal{D}_{x}^2(s)\bigg).
\end{align*}
\end{lemma}

\begin{remark}
Without the bonding force $\sigma=0$, it is not hard to verify that the maximal velocity $\mathcal{V}$ decreases in time. However, this property is not expected in our model due to the bonding force term $\sigma>0$. Instead, we use the modulus preservation property of the rotation operator $R$ as in Proposition~\ref{prop:tan} to get the uniform estimate for the maximal velocity.
\end{remark}

\begin{proof}[Proof of Lemma~\ref{lemma 4.1}]
For a fixed $t>0$, we can take an index $i_t$ such that
\[\|v_{i_t}(t)\|=\max_{1\leq k\leq N }\|v_k(t)\|.\]
Then, by \eqref{mainib}, we have
\begin{align*}
\frac{d}{dt}\|v_{i_t}\|^2&=2\sum_{k=1}^N\frac{\psi_{i_t k}}{N}(\langle R_{x_k \rightarrow x_{i_t}} v_k,v_{i_t}\rangle - \|v_{i_t}\|^2) + \frac{2\sigma}{N} \sum_{k=1}^N \langle x_k ,v_{i_t}\rangle \\
&=2\sum_{k=1}^N\frac{\psi_{i_t k}}{N}(\langle R_{x_k \rightarrow x_{i_t}} v_k,v_{i_t}\rangle - \|v_{i_t}\|^2) + \frac{2\sigma}{N} \sum_{k=1}^N \langle x_k-x_{i_t} ,v_{i_t}\rangle
\\
&\leq 2\sum_{k=1}^N\frac{\psi_{i_t k}}{N}(\| R_{x_k \rightarrow x_{i_t}} v_k\|\|v_{i_t}\| - \|v_{i_t}\|^2) + \frac{2\sigma}{N} \sum_{k=1}^N \langle x_k-x_{i_t} ,v_{i_t}\rangle
.
\end{align*}
We use the modulus conservation property in Proposition \ref{prop:tan} to obtain
\begin{align*}
\frac{d}{dt}\|v_{i_t}\|^2
&\leq 2\sum_{k=1}^N\frac{\psi_{i_t k}}{N}(\|  v_k\|\|v_{i_t}\| - \|v_{i_t}\|^2) + \frac{2\sigma}{N} \sum_{k=1}^N \langle x_k-x_{i_t} ,v_{i_t}\rangle
.
\end{align*}
Note that $\|  v_k\|\|v_{i_t}\| - \|v_{i_t}\|^2\leq 0$. By the assumption of $\psi_m$ and the index $i_t$, we have
\begin{align*}
\frac{d}{dt}\|v_{i_t}\|^2
&\leq 2\sum_{k=1}^N\frac{\psi_m}{N}(\| v_k\|\|v_{i_t}\| - \|v_{i_t}\|^2) + \frac{2\sigma}{N} \sum_{k=1}^N \langle x_k-x_{i_t} ,v_{i_t}\rangle.
\end{align*}
Young's inequality implies that
\begin{align*}
\frac{d}{dt}\|v_{i_t}\|^2
&\leq \sum_{k=1}^N\frac{\psi_m}{N}\bigg(\|v_k\|^2 -\|v_{i_t}\|^2\bigg) + \frac{\sigma}{N} \sum_{k=1}^N \bigg(\frac{ \|x_k-x_{i_t}\|^2}{\eta} +\eta\|v_{i_t}\|^2\bigg),
\end{align*}
for any $\eta>0$. Thus, we have
\begin{align*}
\frac{d}{dt}\|v_{i_t}\|^2\leq  \psi_m\E_K- (\psi_m-\sigma \eta )\|v_{i_t}\|^2 + \frac{\sigma}{N} \sum_{k=1}^N \frac{ \|x_k-x_{i_t}\|^2}{\eta} .
\end{align*}

Let $\eta=\psi_m/2\sigma$. Then
\begin{align*}
\frac{d}{dt}\|v_{i_t}\|^2&\leq  \psi_m\E_K- \frac{\psi_m}{2}\|v_{i_t}\|^2 + \frac{2\sigma^2}{\psi_m} \max_{1\leq i,j\leq N}\|x_{i_t}-x_j\|^2.
\end{align*}
Therefore, we have
\begin{align*}
\|v_{i_t}(t)\|^2&\leq e^{-\frac{\psi_m}{2}t}\|v_{i_t}(0)\|^2+e^{ -\frac{\psi_m}{2}t}\int_0^t e^{ \frac{\psi_m}{2}s}\bigg( \psi_m\E_K(s)+ \frac{2\sigma^2}{\psi_m} \max_{1\leq i,j\leq N}\|x_{i_s}-x_j\|^2(s)\bigg)ds.
\end{align*}
This implies that
\begin{align*}
\|v_{i_t}(t)\|^2&\leq e^{-\frac{\psi_m}{2}t}\|v_{i_t}(0)\|^2+(1-e^{-\frac{\psi_m}{2}t})
\bigg( 2 \sup_{0\leq s\leq t }\E_K(s)+ \frac{4\sigma^2}{\psi_m^2} \sup_{0\leq s\leq t }\max_{1\leq i,j\leq N}\|x_{i_s}-x_j\|^2(s)\bigg)
.\end{align*}
\end{proof}

Next, we provide an estimate for inhomogeneous term $F^{ij}$ via $\mathcal{D}_x$, $\mathcal{D}_v$ and $\mathcal{V}$.
\begin{lemma}\label{lemma 4.3}
Let $\xvi$ be the solution to \eqref{maino}. We assume that $\psi_{ij}$ satisfies $(\mathcal{H}1)$-$(\mathcal{H}3)$. Then the following estimates hold.
\begin{align*}
|F^{ij}_{2}(t)|&\leq (\mathcal{V}(t)+6\|\psi\|_{\mathcal{C}^1} )\mathcal{V}(t) \mathcal{D}_x^2(t)+\psi_0\mathcal{V}(t)\mathcal{D}_x^3(t)
+\frac{\sigma}{2}\mathcal{D}_x^4(t),\end{align*}
and
\begin{align*}
|F^{ij}_{3}(t)|\leq (3\mathcal{V}(t)+6\|\psi\|_{\mathcal{C}^1}+2\sigma)\mathcal{V}(t)\mathcal{D}_{x}^2(t)+(3\mathcal{V}(t)+7\|\psi\|_{\mathcal{C}^1})\mathcal{V}(t)\mathcal{D}_{v}^2(t)  +
\psi_0\mathcal{V}(t)\mathcal{D}_{x}^4(t),
\end{align*}
where $\psi_0=\psi(0)$ and
\[\|\psi\|_{\mathcal{C}^1}=\sup_{x\in [0,2]}(|\psi(x)|+|\psi'(x)|).\]
\end{lemma}
\begin{proof}
For simplicity, we define
\begin{align*}
F^{ij}_{21}:=&
-\frac{\|v_i\|^2+\|v_j\|^2}{2}\|x_i-x_j\|^2,
\\
F^{ij}_{22}:=& \frac{\psi_0}{N}\sum_{k=1}^N
\langle R_{x_k \rightarrow x_i} (v_k)
-R_{x_k \rightarrow x_j} (v_k)
,x_i-x_j\rangle,
\\
F^{ij}_{23}:=&
\sum_{k=1}^N\left (\frac{\psi_{ik}}{N}-\frac{\psi_{ii}}{N}\right)\langle R_{x_k \rightarrow x_i} (v_k)- v_i
,x_i-x_j\rangle
,\\
F^{ij}_{24}:=&
\sum_{k=1}^N\left (\frac{\psi_{jk}}{N}-\frac{\psi_{jj}}{N}\right)\langle R_{x_k \rightarrow x_j} (v_k)- v_j
,x_i-x_j\rangle
,\\
F^{ij}_{25}:=&
 \frac{\sigma}{4N}\sum_{k=1}^N\|x_k-x_i\|^2 \|x_i-x_j\|^2
, \\
 F^{ij}_{26}:=&
 \frac{\sigma}{4N}\sum_{k=1}^N\|x_k-x_j\|^2 \|x_i-x_j\|^2.
 \end{align*}
Then
\begin{align*}
   F^{ij}_2(t)=F^{ij}_{21}(t)+F^{ij}_{22}(t)+F^{ij}_{23}(t)+F^{ij}_{24}(t)+F^{ij}_{25}(t)+F^{ij}_{26}(t).
 \end{align*}
Clearly, we have
\begin{align*}
|F^{ij}_{21}|\leq \mathcal{V}^2 \|x_i-x_j\|^2.
\end{align*}

Note that
\begin{align}\label{est:F22}
|F^{ij}_{22}| &\leq \frac{\psi_0}{N}\sum_{k=1}^N
\|R_{x_k \rightarrow x_i} (v_k)
-R_{x_k \rightarrow x_j} (v_k)\|
\|x_i-x_j\|.
\end{align}
By the definition of the rotation operator $R$ and $\|x_i\|=1$ for any $i\in \{1,\ldots,N\}$ and $t\geq 0$, we have
\begin{align*}\|R_{x_k \rightarrow x_i} (v_k)
-R_{x_k \rightarrow x_j} (v_k)\| &\leq
\|\langle x_k,x_i-x_j\rangle  v_k\|  +\| \langle x_i-x_j,v_k\rangle x_k\|
\\
&\quad+ \left\|(1-  \langle x_k,x_i\rangle) \left( \frac{x_k \times x_i}{|x_k \times x_i|} \right) \left( \frac{\langle x_k \times x_i, v_k\rangle }{|x_k \times x_i|} \right)\right\|
\\
&\quad+\left\| (1-  \langle x_k,x_j\rangle) \left( \frac{x_k \times x_j}{|x_k \times x_j|} \right) \left( \frac{\langle x_k \times x_j,v_k\rangle}{|x_k \times x_j|} \right)  \right\|
\\&
\leq 2\mathcal{V}\|x_i-x_j\|+\frac{\mathcal{V}}{2}\|x_i-x_k\|^2+\frac{\mathcal{V}}{2}\|x_j-x_k\|^2.
\end{align*}
This and \eqref{est:F22} yield that
\begin{align*}
|F^{ij}_{22}| &\leq
2 \psi_0\mathcal{V} \|x_i-x_j\|^2+\frac{\psi_0\mathcal{V}}{2N}\sum_{k=1}^N
\|x_i-x_j\|\|x_i-x_k\|^2+\frac{\psi_0\mathcal{V}}{2N}\sum_{k=1}^N\|x_i-x_j\|\|x_j-x_k\|^2.
\end{align*}
Note that $|\psi_{ik}-\psi_{ii}|=|\psi(\|x_i-x_k\|)-\psi(0)|\leq \|\psi\|_{\mathcal{C}^1}\|x_i-x_k\|$. By the modulus conservation property of the rotation operator $R$ and this estimate,
\begin{align*}
|F^{ij}_{23}| &\leq\sum_{k=1}^N \frac{ \|\psi\|_{\mathcal{C}^1}}{N}\big(\|R_{x_k \rightarrow x_i} (v_k)\|+\| v_i\|\big)
\|x_i-x_k\|\|x_i-x_j\|\\
&\leq \frac{ 2\|\psi\|_{\mathcal{C}^1}}{N}   \mathcal{V}\sum_{k=1}^N  \|x_i-x_k\|\|x_i-x_j\|.
\end{align*}
Similarly,
\begin{align*}
|F^{ij}_{24}|
&\leq \frac{2 \|\psi\|_{\mathcal{C}^1}}{N}\mathcal{V}\sum_{k=1}^N  \|x_j-x_k\|\|x_i-x_j\|.
\end{align*}
Therefore,
\begin{align*}
|F^{ij}_{2}(t)|&\leq (\mathcal{V}(t)+6\|\psi\|_{\mathcal{C}^1} )\mathcal{V}(t) \mathcal{D}_x^2(t)+\psi_0\mathcal{V}(t)\mathcal{D}_x^3(t)
+\frac{\sigma}{2}\mathcal{D}_x^4(t).\end{align*}

Similarly, we define
 \begin{align*}
 F^{ij}_{31}=&2\langle -\|v_i\|^2x_i
 +\|v_j\|^2x_j
 ,v_i-v_j\rangle,
\\
F^{ij}_{32}=&
2\sum_{k=1}^N\frac{\psi_0}{N} \left\langle R_{x_k \rightarrow x_i} (v_k)
-R_{x_k \rightarrow x_j} (v_k)
,v_i-v_j\right\rangle,
\\
F^{ij}_{33}=&
2\sum_{k=1}^N\left(\frac{\psi_{ik}}{N}-\frac{\psi_{ii}}{N}\right)\left\langle R_{x_k \rightarrow x_i} (v_k)- v_i
,v_i-v_j\right\rangle,
\\
F^{ij}_{34}=&-2\sum_{k=1}^N\left(\frac{\psi_{jk}}{N}-\frac{\psi_{jj}}{N}\right)\left\langle R_{x_k \rightarrow x_j} (v_k)- v_j
,v_i-v_j\right\rangle,
\\
F^{ij}_{35}=&
\frac{2\sigma}{N}\sum_{k=1}^N  ( \langle x_i,x_k \rangle-1) \langle  x_i,v_j\rangle
,\\
F^{ij}_{36}=&
\frac{2\sigma}{N}\sum_{k=1}^N
 (\langle x_j,x_k \rangle-1) \langle  x_j
,v_i\rangle.
 \end{align*}
Then,
\begin{align*}
 F^{ij}_3(t) =F^{ij}_{31}+F^{ij}_{32}+F^{ij}_{33}+F^{ij}_{34}+F^{ij}_{35}+F^{ij}_{36}.
 \end{align*}

We next provide upper bounds for each of the terms sequentially.
\begin{align*}
   F^{ij}_{31}&=2 \|v_i\|^2\langle x_i,v_j\rangle
 +2\|v_j\|^2\langle x_j
 ,v_i\rangle
 \\
 &=2 (\|v_i\|^2-\|v_j\|^2)\langle x_i,v_j\rangle
 +2\|v_j\|^2(\langle x_i,v_j\rangle+\langle x_j
 ,v_i\rangle)
\\
 &=2 (\|v_i\|^2-\|v_j\|^2)\langle x_i-x_j,v_j\rangle
 -2\|v_j\|^2\langle x_i-x_j,v_i-v_j\rangle.
 \end{align*}
Therefore, we have
\begin{align*}
  | F^{ij}_{31}|
 &\leq 2 \|v_i-v_j\|(\|v_i\|+\|v_j\|)\|x_i-x_j\|\|v_j\|
 +2\|v_j\|^2\|x_i-x_j\| \|v_i-v_j\|\\
 &\leq 6\mathcal{V}^2\|x_i-x_j\| \|v_i-v_j\|.
 \end{align*}
Similar to $F^{ij}_2$, we can obtain
\begin{align*}
  | F^{ij}_{32}|
 &\leq 4\psi_0\mathcal{V}\|x_i-x_j\|\|v_i-v_j\|
 \\&\quad+
 \frac{\psi_0\mathcal{V}}{N}\sum_{k=1}^N
 \|x_i-x_k\|^2\|v_i-v_j\|+
 \frac{\psi_0\mathcal{V}}{N}\sum_{k=1}^N\|x_j-x_k\|^2\|v_i-v_j\|,\\
|F^{ij}_{33}|
&\leq \frac{ 4\|\psi\|_{\mathcal{C}^1}\mathcal{V}}{N}\sum_{k=1}^N  \|x_i-x_k\|\|v_i-v_j\|,
\\
|F^{ij}_{34}|
&\leq \frac{ 4\|\psi\|_{\mathcal{C}^1}\mathcal{V}}{N}\sum_{k=1}^N  \|x_j-x_k\|\|v_i-v_j\|.
\end{align*}
For $F^{ij}_{35}$ and $F^{ij}_{36}$,
\begin{align*}
|F^{ij}_{35}|
\leq\frac{\sigma\mathcal{V}}{N}\sum_{k=1}^N   \| x_i-x_k\|^2,
\qquad
|F^{ij}_{36}|
\leq\frac{\sigma\mathcal{V}}{N}\sum_{k=1}^N   \| x_j-x_k\|^2.
\end{align*}
Therefore,
\begin{align*}
|F^{ij}_{3}(t)|\leq (3\mathcal{V}(t)+6\|\psi\|_{\mathcal{C}^1}+2\sigma)\mathcal{V}(t)\mathcal{D}_{x}^2(t)+(3\mathcal{V}(t)+7\|\psi\|_{\mathcal{C}^1})\mathcal{V}(t)\mathcal{D}_{v}^2(t)  +
\psi_0\mathcal{V}(t)\mathcal{D}_{x}^4(t).
\end{align*}
\end{proof}

As a consequence of Proposition \ref{prop:vt}, Proposition \ref{prop 4.2}, Lemma \ref{lemma 4.1}, and Lemma \ref{lemma 4.3}, we prove our main theorem as follows:

\begin{proof}[Proof of Theorem \ref{thm:1}.]
By Proposition \ref{prop 4.2}, $X^{ij}$ satisfies
\begin{align}\label{eq:X_proof}
\frac{d}{dt}X^{ij}(t)=A X^{ij}+F^{ij},
 \end{align}
where  the coefficient matrix $A$ is given by
 \begin{align*}
A =
\left(\begin{matrix}
0 & 2 &0  \\
-\sigma & -\psi_0  &1 \\
0 & -2\sigma & -2\psi_0
\end{matrix}\right).
\end{align*}
Note that the eigenvalues of $A$ is
\[\left\{- \psi_0 , ~- \psi_0 -\sqrt{\psi_0^2-4\sigma},~- \psi_0 +\sqrt{\psi_0^2-4\sigma}\right\}.\]

Thus, if $\psi_0^2 \leq 4\sigma$, then  the maximum real part of the eigenvalues of $A(t)$ is $-\psi_0$.
If $\psi_0^2> 4\sigma$, then the maximum real part of the eigenvalues of $A(t)$ is $-\psi_0+\sqrt{\psi_0^2-4\sigma}$ and satisfies
\begin{align*}- \psi_0 +\sqrt{\psi_0^2-4\sigma}
&=-\frac{4\sigma}{ \psi_0 +\sqrt{\psi_0^2-4\sigma}}
\\
&\leq-\frac{2\sigma}{ \psi_0}.\end{align*}

We denote the maximum real part of the eigenvalues of $A$ by
 \[-\mu<0.\]
 Therefore, by \eqref{eq:X_proof},
\begin{align*}\frac{1}{2}\frac{d}{dt}\|X^{ij}(t)\|^2&=\langle X^{ij}(t), A X^{ij}(t)\rangle  +\langle X^{ij}(t), F^{ij}(t)\rangle
\\
&\leq -\mu \|X^{ij}(t)\|^2 +\| X^{ij}(t)\|~\| F^{ij}(t)\|.
\end{align*}
Thus, we have
\begin{align*}\frac{d}{dt}\|X^{ij}(t)\|
&\leq -\mu \|X^{ij}(t)\| +\| F^{ij}(t)\|.
\end{align*}
By Lemma \ref{lemma 4.3},
\begin{align*}\| F^{ij}(t)\|\leq C\frac{\mathcal{V}(t)+ \mathcal{V}^2(t)}{4}\mathcal{D}_x^2(t)+C\frac{\mathcal{V}(t)+ \mathcal{V}^2(t)}{4}\mathcal{D}_v^2(t)
+\frac{\sigma}{2}\mathcal{D}_x^4(t),\end{align*}
where $C=C(\psi,\sigma)$ is a constant depending on $\psi$ and $\sigma$.

Therefore, we have
\begin{align*}\frac{d}{dt}\|X^{ij}(t)\|
&\leq -\mu \|X^{ij}(t)\| +C\frac{\mathcal{V}(t)+ \mathcal{V}^2(t)}{4} \mathcal{D}_x^2(t)+C\frac{\mathcal{V}(t)+ \mathcal{V}^2(t)}{4}\mathcal{D}_v^2(t)
+\frac{\sigma}{2}\mathcal{D}_x^4(t).
\end{align*}

Let
\[X(t)=\max_{1\leq i,j\leq N}  \|X^{ij}(t)\|.\]
Clearly, we have
\begin{align*}
\mathcal{D}_x^2(t),~\mathcal{D}_v^2(t)\leq X(t).
\end{align*}
Then
\begin{align}\label{est:X}\frac{d}{dt}X(t)
&\leq -\mu X(t) +C\frac{\mathcal{V}(t)+ \mathcal{V}^2(t)}{2}X(t)+\frac{\sigma}{2}X^2(t).
\end{align}

We let
\begin{align*}
\mathcal{V}_0:=\left\{ \begin{array}{ll}
             \displaystyle  \frac{\mu}{4C}&\mbox{if}\quad\displaystyle \frac{\mu}{4C}<1,  \\
               \displaystyle\sqrt{\frac{\mu}{4C}}&\mbox{if}\quad\displaystyle \frac{\mu}{4C}\geq 1,
              \end{array}\right.\qquad
                            \E_0:=
   \left\{ \begin{array}{ll}
              \displaystyle \frac{\mu^2}{64C^2}&\mbox{if}\quad\displaystyle \frac{\mu}{4C}<1,  \\
               \displaystyle\frac{\mu}{16C}&\mbox{if}\quad\displaystyle \frac{\mu}{4C}\geq 1,
              \end{array}\right.
\end{align*}
and
 $\psi_m=\psi(\sqrt{X_M})$, where  $X_M>0$ is a constant satisfying
\begin{align}\label{X_M}
  \left\{ \begin{array}{ll}
               \displaystyle  \sqrt{X_M}=\frac{\mu}{\sqrt{128 }C\sigma }\psi(\sqrt{X_M})&\mbox{if}\quad  \displaystyle \frac{\mu}{4C}<1,  \\
               \displaystyle  \sqrt{X_M}=\frac{\sqrt{\mu}}{\sqrt{32 C}\sigma}\psi(\sqrt{X_M})&\mbox{if}\quad\displaystyle \frac{\mu}{4C}\geq 1.
              \end{array}\right.
\end{align}
We assume that
\begin{align}\label{eq 4.4}
\mathcal{V}(0)<\mathcal{V}_0,
\quad
   \E(0)<
   \E_0,\quad X(0)< \min\left\{\frac{\mu}{2\sigma},  X_M
 \right\}.
\end{align}
By the initial data assumption, there is $\epsilon>0$ such that for $t\in [0,\epsilon)$,
\begin{align*}
   \mathcal{V}(t)<\mathcal{V}_0,\quad X(t)< \min\left\{\frac{\mu}{2\sigma},  X_M
 \right\}.
\end{align*}
 Assume that there is $T>0$ such that
\begin{align}\label{assumption:T1}
   \mathcal{V}(t)< \mathcal{V}_0,\quad X(t)< \min\left\{\frac{\mu}{2\sigma},  X_M
 \right\}\quad \mbox{for any $t\in [0,T)$},
\end{align}
but
\begin{align}\label{assumption:T2}
   \mathcal{V}(T)= \mathcal{V}_0,\quad\mbox{or}\quad X(T)= \min\left\{\frac{\mu}{2\sigma},  X_M
 \right\}.
\end{align}
Then by \eqref{est:X}, for $t\in [0,T)$,
\begin{align*}\frac{d}{dt}X(t)
&\leq -\frac{\mu}{2} X(t),
\end{align*}
i.e., for $t\in [0,T)$,
\[X(t)<X(0).\]
 This implies that
\[X(T)=\lim_{t\to T^-}X(t)\leq X(0)<\min\left\{\frac{\mu}{2\sigma},  X_M
 \right\}.\]

Note that for any $i,j \in \{1,\ldots,N\}$, $\psi_{ij}(s)=\psi(\|x_i(s)-x_j(s)\|)\geq \psi(\sqrt{X_M})= \psi_m$ on $0\leq s\leq T$ and  by Proposition \ref{prop:vt}, for any $t\geq 0$,
\begin{align*}
 \E_K(t) \leq  \E(t)\leq  \E(0).
\end{align*}
Here, $\E_K(t)$ is given in \eqref{eqn:e}.  Therefore, by Lemma \ref{lemma 4.1} and the above,
\begin{align}\begin{aligned}\label{est:V}
\mathcal{V}^2(t)&\leq e^{-\frac{\psi_m}{2}t}\mathcal{V}^2(0)+(1-e^{-\frac{\psi_m}{2}t})
\bigg(  2\sup_{0\leq s\leq t }\E_K(s)+ \frac{4\sigma^2}{\psi_m^2} \sup_{0\leq s\leq t }\mathcal{D}_{x}^2(s)\bigg)
\\
&\leq e^{-\frac{\psi_m}{2}t}\mathcal{V}^2(0)+(1-e^{-\frac{\psi_m}{2}t})
\bigg(2 \E(0)+ \frac{4\sigma^2}{\psi_m^2} \sup_{0\leq s\leq t }X(s)\bigg).\end{aligned}\end{align}
By \eqref{est:V}, for $t\in [0,T)$,
\begin{align*}
\mathcal{V}^2(t)
&\leq e^{-\frac{\psi_m}{2}t}\mathcal{V}^2(0)+(1-e^{-\frac{\psi_m}{2}t})
\bigg( 2\E(0)+ \frac{4\sigma^2}{\psi_m^2} X(0)\bigg)
\end{align*}
and
\[\mathcal{V}^2(T)=\lim_{t\to T^-}\mathcal{V}^2(t)\leq\max\left\{\mathcal{V}^2(0), ~ 2\E(0)+ \frac{4\sigma^2}{\psi_m^2} X(0)   \right\}<\left\{ \begin{array}{ll}
              \displaystyle  \frac{\mu^2}{16C^2}&\mbox{if}\quad\displaystyle \frac{\mu}{4C}<1,  \\
             \displaystyle   \frac{\mu}{4C}&\mbox{if}\quad\displaystyle \frac{\mu}{4C}\geq 1.
              \end{array}\right.\]
Here, we used the equality in \eqref{X_M}.

Therefore, there is no $T>0$ satisfying \eqref{assumption:T1} and \eqref{assumption:T2}, i.e.,
for any $t\geq 0$,  the followings hold.
\begin{align*}
   \mathcal{V}(t)<\mathcal{V}_0,\quad X(t)<\min\left\{\frac{\mu}{2\sigma},  X_M
 \right\}.
\end{align*}
Therefore, by \eqref{est:X},
\begin{align*}\frac{d}{dt}X(t)
&\leq -\frac{\mu}{2} X(t), \quad \mbox{for any $t\in [0,T)$}.
\end{align*}
By Gronwall's lemma, we obtain the desired result.
\end{proof}

\section{Numerical simulations}\label{sec:5}
\setcounter{equation}{0}
In this section, we conduct some numerical simulations of $6$-agents system of \eqref{maino} to confirm our mathematical results in Theorem \ref{thm:1} and to check the exponential convergence rate
\[\delta=\frac{\mu}{2}.\]
We use the fourth order Runge-Kutta method and MATLAB programming for the simulations.\\

\begin{figure}[!ht]
\centering
\begin{minipage}{0.3\textwidth}
\centering
\includegraphics[width=\textwidth]{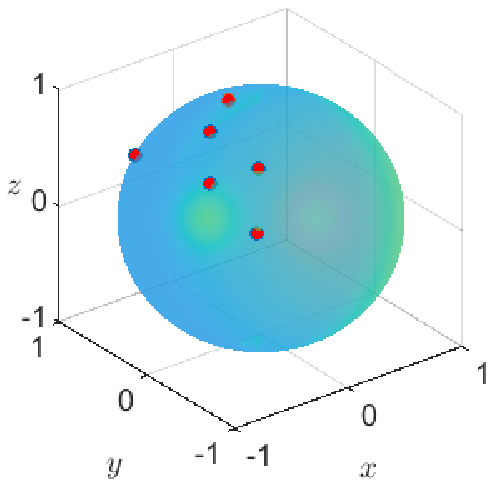}\\
(a) $t=0$
\end{minipage}
\begin{minipage}{0.3\textwidth}
\centering
\includegraphics[width=\textwidth]{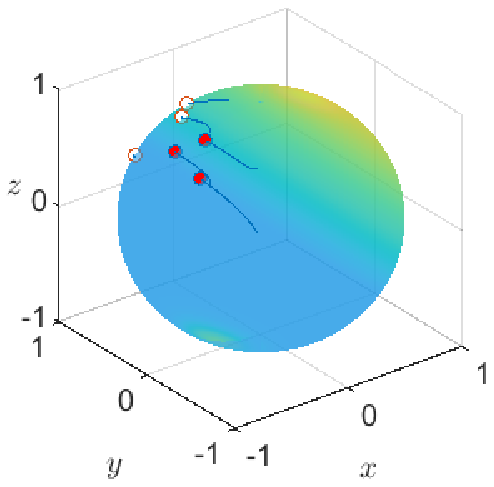}\\
(b) $t=1$
\end{minipage}
\begin{minipage}{0.3\textwidth}
\centering
\includegraphics[width=\textwidth]{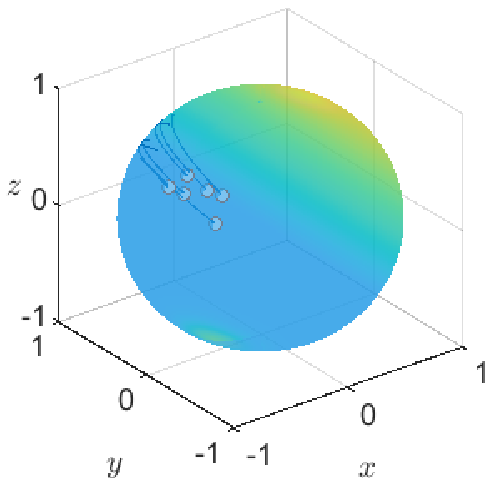}\\
(c) $t=5$
\end{minipage}
\vfill
\begin{minipage}{0.3\textwidth}
\centering
\includegraphics[width=\textwidth]{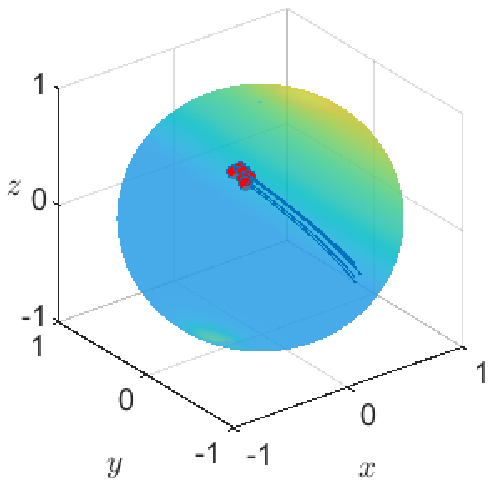}\\
(d) $t=15$
\end{minipage}
\begin{minipage}{0.3\textwidth}
\centering
\includegraphics[width=\textwidth]{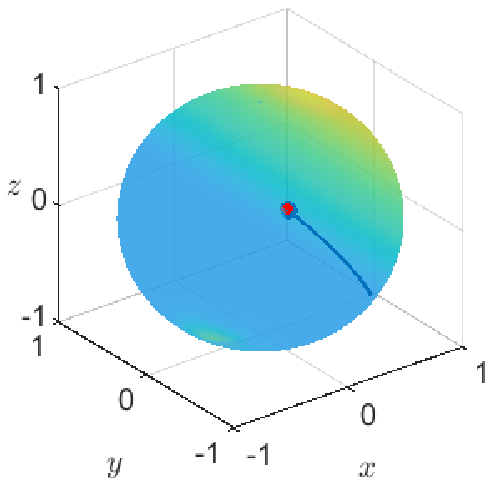}\\
(e) $t=30$
\end{minipage}
\begin{minipage}{0.3\textwidth}
\centering
\includegraphics[width=\textwidth]{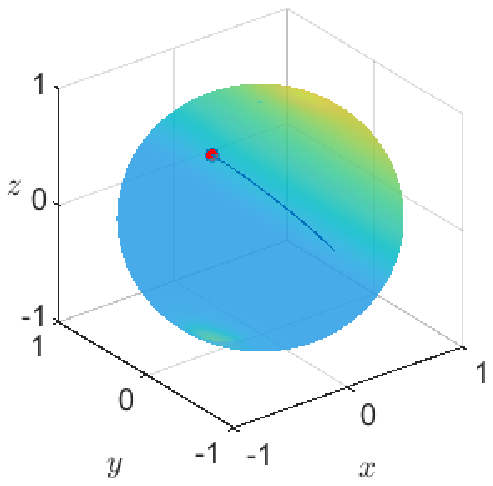}\\
(f) $t=80$
\end{minipage}
\caption{Time evolution of solution of \eqref{maino} under the admissible initial data condition}
\label{fig1}
\end{figure}

Let
\[\psi(x)=3(\exp(2-x)-1) \quad\text{and}\quad \sigma=1.\]
Then we can easily check that $\psi$ satisfies the condition in $(\mathcal{H}1)$-$(\mathcal{H}3)$.
The initial configuration is randomly chosen satisfying conditions in Theorem \ref{thm:1} such that
\begin{align*}
&x_1(0)=(-0.3903,  -0.4756, \phantom{-}0.7883 ), && x_2(0)=(-0.5800,  -0.7067,\phantom{-} 0.4052 ),\\
&x_3(0)=(-0.6746,  -0.2998, \phantom{-}0.6746  ), && x_4(0)=( -0.4472, \phantom{-}0.0000, \phantom{-}0.8944 ),\\
&x_5(0)=(-0.1249,  \phantom{-} 0.2084, \phantom{-}0.9700 ), && x_6(0)=(-0.6236,\phantom{-}0.6236,  \phantom{-}0.4714),
\end{align*}
and
\begin{align*}
&v_1(0)=(-0.4707, \phantom{-}0.1259, -0.1571), && v_2(0)=(-0.0986, \phantom{-}0.4355,  \phantom{-}0.6185),  \\
&v_3(0)=(\phantom{-}0.1892, \phantom{-}0.1666, \phantom{-}0.2631)  ,&& v_4(0)=(\phantom{-}0.4605, \phantom{-}0.5046, \phantom{-}0.2302),  \\
&v_5(0)=(-0.4914, \phantom{-}0.7722, -0.2292), && v_6(0)=(-0.0148, \phantom{-}0.1342,  -0.1971).
\end{align*}

The time evolution of the solution to \eqref{maino} under the above setting  is given in Figure \ref{fig1}.  To visually represent the solution,  we here use  red points for the agent's positions $\{x_i(t)\}_{i=1}^N$   at $t=t_0$ and the blue lines for the trajectory of agents on the time interval  $[t_0-3, 3]$. Here the white points in Figure  \ref{fig1}(b) and (c) mean the agent's positions on the opposite side of the visible side. Eventually, we can observe the phenomenon in Figure  \ref{fig1} that all the agents gather to one point and they converge into a trajectory orbiting a great circle at a constant speed.

In Figure \ref{fig2}, we display the maximal spatial diameter $\max_{i,j} \|x_i(t)-x_j(t)\|$ of the solution and we can check that it decays exponentially as we proved in Theorem \ref{thm:1}. Additionally, if we increase the inter-particle bonding force such as $\sigma=5$, then the above initial data does not satisfy the admissible condition, i.e.,
\[X(0)>  \min\left\{\frac{\mu}{2\sigma},  X_M
 \right\}.\]
\begin{figure}[!ht]
\centering
\begin{minipage}{0.45\textwidth}
\centering
\includegraphics[width=\textwidth]{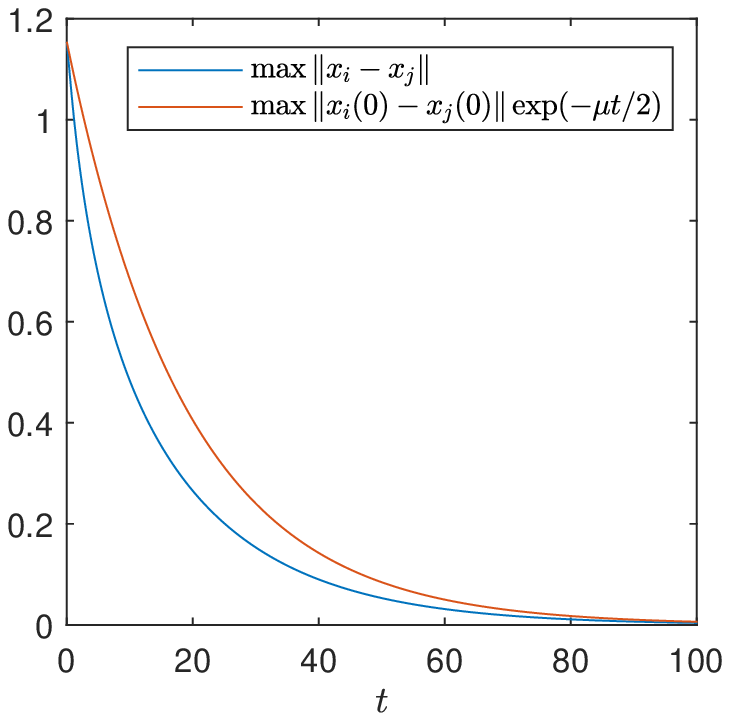}\\
(a) $\max_{i,j} \|x_i(t)-x_j(t)\|$
\end{minipage}
\begin{minipage}{0.45\textwidth}
\centering
\includegraphics[width=\textwidth]{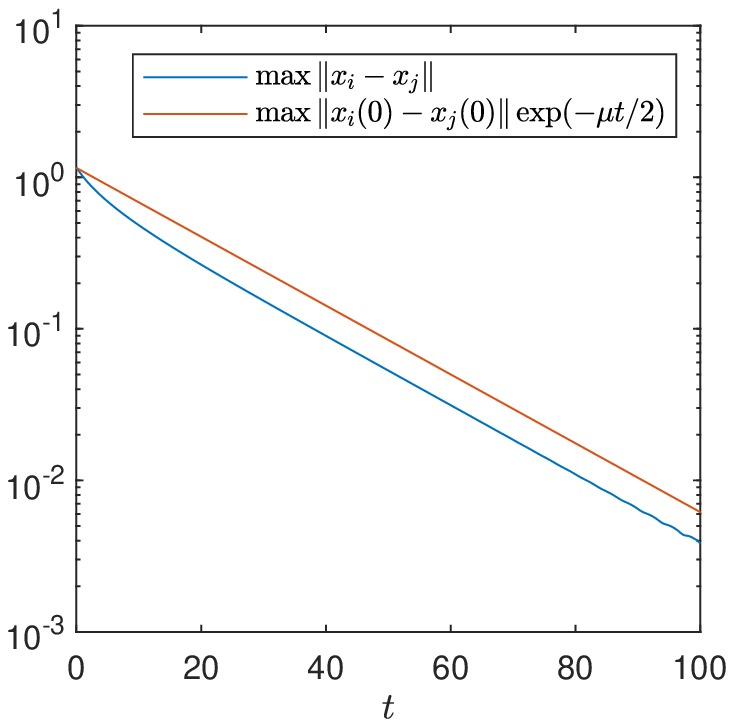}\\
(b) Semi-log graph of (a)
\end{minipage}
\caption{Maximum position diameter for the solution satisfying \eqref{eq 4.4}}
\label{fig2}
\end{figure}

\begin{figure}[!ht]
\centering
\begin{minipage}{0.45\textwidth}
\centering
\includegraphics[width=\textwidth]{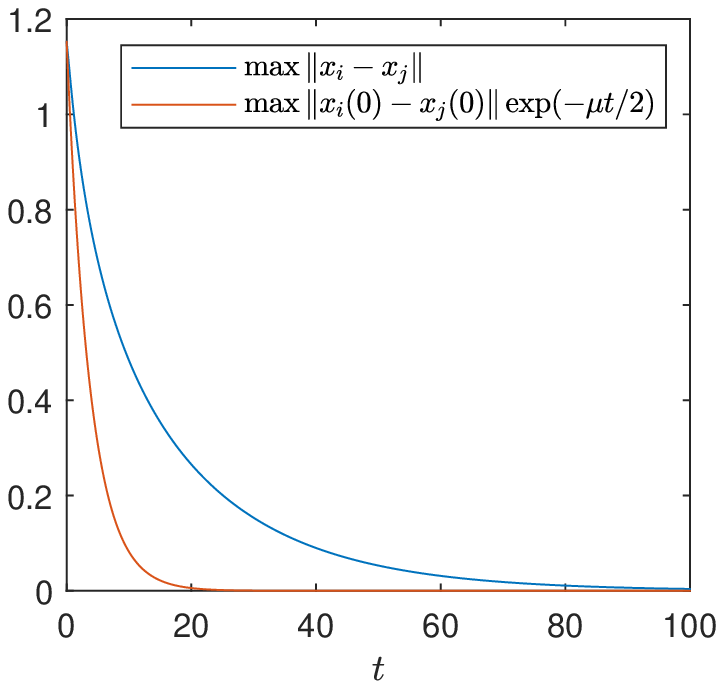}\\
(a) $\max_{i,j} \|x_i(t)-x_j(t)\|$
\end{minipage}
\begin{minipage}{0.45\textwidth}
\centering
\includegraphics[width=\textwidth]{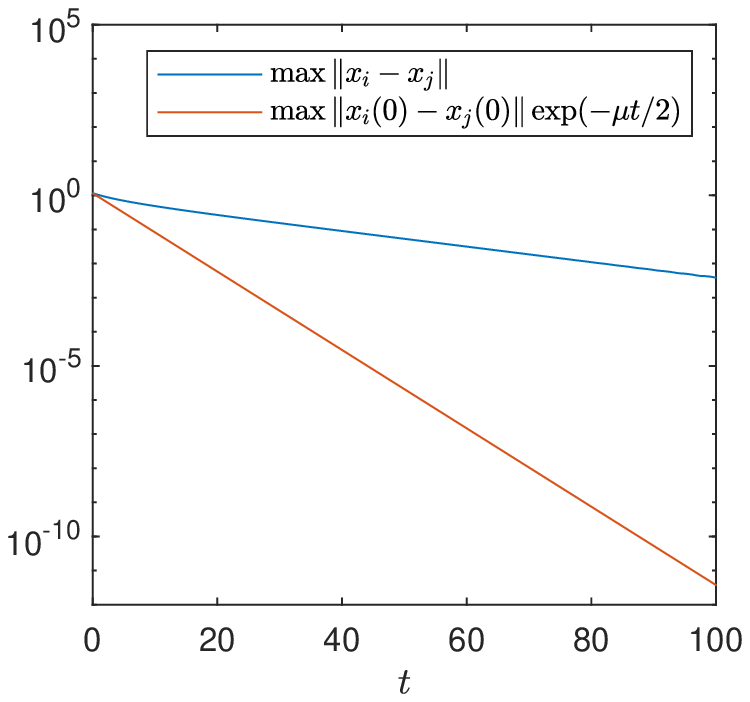}\\
(b) Semi-log graph of (a)
\end{minipage}
\caption{Maximum position diameter for the solution not satisfying \eqref{eq 4.4}.}
\label{fig3}
\end{figure}

Then, we can see that the sufficient exponential decay rate does not appear. See Figure \ref{fig3}.

\section{conclusion and discussion}\label{sec:6}
\setcounter{equation}{0}

In this paper, we studied the interactions of  the inter-particle bonding forces and flocking operator on a sphere. We  show that the model has the complete position  flocking for an admissible initial condition depending on $\psi$, $\sigma$. We note that the initial data condition does not depend on the number of particles $N$. As the flat space case, we obtain that the ensemble converges to  one point particle with one velocity when a flocking model has the inter-particle bonding forces. We crucially use the energy dissipation property in Proposition \ref{prop:vt} because our model has no momentum conservation. From the energy dissipation leads to  uniform upper bound of velocities. We simultaneously use the Lyapunov functional method and a reduction to a linearized system of differential equations to obtain the asymptotic position alignment result.

\section*{Acknowledgments}
S.-H. Choi is partially supported by NRF of Korea (no. 2017R1E1A1A03070692) and Korea Electric Power Corporation(Grant number: R18XA02).



\end{document}